\RequirePackage{snapshot}
\newif\ifproofs\proofsfalse\ifproofs\RequirePackage[displaymath,mathlines]{lineno}\fi

\newif\ifsubsections
\documentclass[]{pcmi}

\usepackage[sc]{mathpazo}          % Palatino with smallcaps as text font
\usepackage{eulervm}               % Euler math
\usepackage[mathscr]{eucal}		   % provide proper \mathscr
\usepackage[scaled=0.86]{berasans} % Bera for san serif family
\usepackage[scaled=1]{inconsolata} % Inconsolata for fixed width
\usepackage[T1]{fontenc}

\usepackage[%
	protrusion=true,
	expansion=false,
	auto=false
	]{microtype}

\usepackage{xcolor}
\usepackage{graphicx}
\graphicspath{{./figures/}}

\ifdraft
    \definecolor{linkred}{rgb}{0.7,0.2,0.2}
    \definecolor{linkblue}{rgb}{0,0.2,0.6}
\else
    \definecolor{linkred}{cmyk}{0.0,0.0,0.0,1.0}
    \definecolor{linkblue}{cmyk}{0,0.0,0.0,1.0}
\fi

\usepackage{pinlabel}
\usepackage{subfigure}
\usepackage{mathdots}
\usepackage{stmaryrd}

\usepackage{tikz}
\usetikzlibrary{arrows}
\usetikzlibrary{decorations.pathreplacing}
\usepackage{verbatim}
\usetikzlibrary{cd}
\usetikzlibrary{calc}
\tikzset{stretch/.initial=1}

\PassOptionsToPackage{hyphens}{url} 
\usepackage[
    setpagesize=false,
    pagebackref,
	pdfpagelabels,
	plainpages=false,
    pdfstartview={FitH 1000},
    bookmarksnumbered=false,
    linktoc=all,
    colorlinks=true,
    anchorcolor=black,
    menucolor=black,
    runcolor=black,
    filecolor=black,
    linkcolor=linkblue,%IM black if not in draft mode
	citecolor=linkblue,%IM black if not in draft mode
	urlcolor=linkred,%IM black if not in draft mode
]{hyperref}%IM
\usepackage[backrefs,msc-links,initials,nobysame]{amsrefs}

\customizeamsrefs 

\theoremstyle{plain}
\newtheorem{theorem}[equation]{Theorem}

\theoremstyle{definition}
\newtheorem{definition}[equation]{Definition}

\theoremstyle{remark}
\newtheorem{remark}[equation]{Remark}
\newtheorem{example}[equation]{Example}
\newtheorem{exercise}[equation]{Exercise}
\newtheorem*{ack}{Acknowledgements}

\newcommand\drawloop[4][]%
   {\draw[shorten <=0pt, shorten >=0pt,#1]
      ($(#2)!\pgfkeysvalueof{/tikz/stretch}!(#2.#3)$)
      let \p1=($(#2.center)!\pgfkeysvalueof{/tikz/stretch}!(#2.north)-(#2)$),
          \n1= {veclen(\x1,\y1)*sin(0.5*(#4-#3))/sin(0.5*(180-#4+#3))}
      in arc [start angle={#3-90}, end angle={#4+90}, radius=\n1]%
   }

\def\F{\mathbb{F}}

\def\Q{\mathbb{Q}}
\def\R{\mathbb{R}}
\def\Z{\mathbb{Z}}

\def\bbT{\mathbb{T}}

\def\bbD{\mathbb{D}}
\def\bbC{\mathbb{C}}

\def\cH{\mathcal{H}}

\def\cM{\mathcal M}
\def\cMhat{\widehat{\mathcal M}}

\def\Sym{\mathrm{Sym}}

\def\CFKR{\CFK_{\F[U, V]}}

\def \gr {\operatorname{gr}}

\def\Cone{\operatorname{Cone}}

\def\d{\partial}

\def\co{\colon}

\def\spinc{\textrm{Spin}^c}
\def\mfs{\mathfrak{s}}
\def\mft{\mathfrak{t}}

\def\id{\textup{id}}

\def\Cone{\operatorname{Cone}}

\def\CF {\mathit{CF}}
\def\HF {\mathit{HF}}
\def\HFK {\mathit{HFK}}
\newcommand\HFKhat{\widehat{\HFK}}
\newcommand\HFKm{\HFK^-}
\newcommand\HFhat{\widehat{\HF}}
\newcommand\CFhat{\widehat{\CF}}
\newcommand\HFp {\HF^+}
\newcommand \CFp {\CF^+}
\newcommand \CFm {\CF^-}
\newcommand \HFm {\HF^-}
\newcommand \CFinf {\CF^{\infty}}

\newcommand \HFinf {\HF^{\infty}}

\def\HFred{\HF^{\operatorname{red}}}

\def\CFK{\mathit{CFK}}

\def\spinc {{\operatorname{spin^c}}}

\def\HFred{\HF_{\operatorname{red}}}

\def\Hom{\mathrm{Hom}}

\def\bfalpha{\boldsymbol{\alpha}}
\def\bfbeta{\boldsymbol{\beta}}
\def\bfx{\mathbf{x}}
\def\bfy{\mathbf{y}}

\newcommand{\lab}[1]{$\scriptstyle #1$}

\title[Heegaard Floer homology]{Lecture notes on Heegaard Floer homology}

\author{Jennifer Hom}
\thanks{The author was partially supported by NSF grant DMS-1552285 and a Sloan Research Fellowship.}
\address {School of Mathematics, Georgia Institute of Technology, Atlanta, GA 30332}
\email{hom@math.gatech.edu}
\begin{document} 
	
\begin{abstract}
These are the lecture notes for a course on Heegaard Floer homology held at PCMI in Summer 2019. We describe Heegaard diagrams, Heegaard Floer homology, knot Floer homology, and the relationship between the knot and 3-manifold invariants. 
\end{abstract} 

%%%%%%%%%%%%%%%%%%%%%%%%%%%%%%%%%%%%%%%%%%%%%%%%%%%%%%%%%%%%%%%%%%%%%
%    
%    Make the title page
%    
%%%%%%%%%%%%%%%%%%%%%%%%%%%%%%%%%%%%%%%%%%%%%%%%%%%%%%%%%%%%%%%%%%%%%
\maketitle
\thispagestyle{empty}

%%%%%%%%%%%%%%%%%%%%%%%%%%%%%%%%%%%%%%%%%%%%%%%%%%%%%%%%%%%%%%%%%%%%%
%    
%    Your lecture notes replace the remainder of this document.
%    
%%%%%%%%%%%%%%%%%%%%%%%%%%%%%%%%%%%%%%%%%%%%%%%%%%%%%%%%%%%%%%%%%%%%%

% Authors: insert the body of your lectures here

\section*{Introduction}

Heegaard Floer homology, defined by Ozsv\'ath-Szab\'o \cite{OS3manifolds1}, and knot Floer homology, defined by Ozsv\'ath-Szab\'o \cite{OSknots} and independently Rasmussen \cite{RasmussenThesis}, are invariants of 3-manifolds and knots inside of them. These notes aim to provide an overview of these invariants and the relationship between them.

We will describe our manifolds and knots via Heegaard diagrams, which we define in Section \ref{sec:Heegaard}. Any two Heegaard diagrams for the same manifold or knot are related by a sequence of moves, called Heegaard moves, much like any two projections for a knot are related by a sequence of Reidemeister moves.

From such diagrams, we will construct chain complexes for 3-manifolds (Section \ref{sec:HF}) and  knots (Section \ref{sec:HFK}) whose chain homotopy type (and in particular, homology) is independent of the choice of Heegaard diagram. The proof that these invariants do not depend on the choice of Heegaard diagram relies on showing invariance under Heegaard moves.

From the knot invariant associated to a knot $K$ in $S^3$, one can compute the 3-manifold invariant for any Dehn surgery along $K$; we discuss this relationship in the case of integral surgery in Section \ref{sec:surgery}. 

%These are notes from a series of lectures on Heegaard Floer homology at PCMI in July 2019. They contain material from the expository articles \cite{OSintro}, \cite{OSsurvey2}, as well as the papers \cite{OS3manifolds1}, \cite{OS3manifolds2}, \cite{OSknots}, \cite{OSinteger}, \cite{RasmussenThesis}, \cite{Zemkeabsgr}.

\begin{ack}
I would like to thank the organizers of the 2019 Park City Mathematics Institute for the opportunity to give the lecture series associated with these notes, and the participants of PCMI for being an attentive audience. I would also like to thank Miriam Kuzbary and Robert Lipshitz for helpful comments on earlier versions of these notes.
\end{ack}

%%%%%%%%%%%%%%%%%%%%%%%%%%%%%%
%%%%%%%%%%%%%%%%%%%%%%%%%%%%%%

\section{Heegaard splittings and diagrams}\label{sec:Heegaard}
\subsection{Heegaard splittings}
Our goal is to define an invariant of closed 3-manifolds and knots inside of them. In order to do this, we will need a way to describe our manifolds and knots. We will do this using Heegaard diagrams. Throughout, we will assume that all of our 3-manifolds are closed and oriented, and that all of our homeomorphisms are orientation preserving, unless stated otherwise. Much of what follows comes from \cite[Lecture 1]{SavelievLectures} and \cite[Sections 2 and 3]{OSintro}.

\begin{definition}
A \emph{handlebody of genus $g$} is a closed regular neighborhood of a wedge of $g$ circles in $\R^3$.
\end{definition}

\begin{definition}\label{defn:Heegaardsplitting}
Let $Y$ be a 3-manifold. A \emph{Heegaard splitting} of $Y$ is a decomposition of $Y = H_1 \cup_f H_2$ where $H_1, H_2$ are handlebodies and $f$ is an orientation reversing homeomorphism from $\d H_1$ to $\d H_2$. The \emph{genus} of the Heegaard splitting is the genus of the surface $\d H_1$ or equivalently $\d H_2$.
\end{definition}

\begin{example}
Note that $S^3 = B^3 \cup B^3$. This is a genus 0 Heegaard splitting of $S^3$.
\end{example}

\begin{example}
Consider $S^3$ as the 1-point compactification of $\R^3$. Consider the circle consisting of the $z$-axis and the point at infinity. A regular neighborhood of this union is a handlebody $H_1$ of genus 1. The complement of $H_1$ is also a handlebody of genus 1. Together, these two handlebodies form a genus 1 Heegaard splitting of $S^3$.
\end{example}

\begin{theorem}
Any closed, orientable 3-manifold $Y$ admits a Heegaard splitting.
\end{theorem}

\begin{proof}
Consider a triangulation of $Y$. Let $H_1$ be a closed regular neighborhood of the 1-skeleton of the triangulation; since $H_1$ is a neighborhood of a graph, it is a handlebody. The complement $H_2 = Y - H_1$ is also a handlebody; namely it is a neighborhood of the dual 1-skeleton, that is, the graph whose vertices are the centers of the tetrahedra and whose edges are segments perpendicular to the faces of the tetrahedra.
\end{proof}

\begin{remark}
Alternatively, a self-indexing Morse function on $Y$ gives rise to a Heegaard splitting of $Y$; see \cite[Section 3]{OSintro}, in particular, Exercise 3.5.
\end{remark}

\begin{definition}
Two Heegaard splittings $Y = H_1 \cup_f H_2$ and $Y = H'_1 \cup_{f'} H'_2$ of $Y$ are \emph{homeomorphic} if there exists a homeomorphism $\phi \co Y \to Y$ taking $H_i$ to $H'_i$. 
\end{definition}

\begin{remark}
One may also consider the following stricter notion of equivalence: two Heegaard splittings $Y = H_1 \cup_f H_2$ and $Y = H'_1 \cup_{f'} H'_2$ of $Y$ are \emph{isotopic} if there exists a map $\psi \co Y \times [0, 1] \to Y$ such that
\begin{enumerate}
	\item $\psi |_{Y \times \{0\}} = \id_Y$,
	\item $\psi |_{Y \times t}$ is a homeomorphism for all $t$,
	\item $\psi |_{Y \times \{1\}}$ sends $H_i$ to $H'_i$.
\end{enumerate}
Note that if $Y = H_1 \cup_f H_2$ and $Y = H'_1 \cup_{f'} H'_2$ are isotopic, then they are homeomorphic (via $\psi |_{Y \times \{1\}}$). The converse is false, since the homeomorphism $\phi \co Y \to Y$ need not be isotopic to the identity.
\end{remark}

\begin{definition}
Let $Y = H_1 \cup_f H_2$ be a genus $g$ Heegaard splitting of $Y$. A \emph{stablization} $H'_1 \cup_{f'} H'_2$ of $H_1 \cup_{f} H_2$ is the genus $g+1$ Heegaard splitting of $Y$ where $H'_1$ consists of $H_1$ together with a neighborhood $N$ of a properly embedded unknotted arc $\gamma$ in $H_2$, and $H'_2$ consists of $H_2 - N$.
\end{definition}

\begin{exercise}
Prove that the homeomorphism type (in fact, isotopy type) of a stabilization of $Y = H_1 \cup_f H_2$ is independent of the choice of $\gamma$.
\end{exercise}

The following theorem of Reidemeister and Singer highlights the importance of stabilizations.

\begin{theorem}[\cite{Reidemeister, Singer}]\label{thm:ReidSing}
Any two Heegaard splittings of $Y$ become isotopic after sufficiently many stablizations.
\end{theorem}

\subsection{Heegaard diagrams}
We will describe a Heegaard splitting via a Heegaard diagram $\cH$, as defined below. 

\begin{definition}
Let $H$ be a handlebody of genus $g$. A \emph{set of attaching circles} for $H$ is a set $\{ \gamma_1, \dots, \gamma_g \}$ of simple closed curves in $\Sigma = \d H$ such that
\begin{enumerate}
	\item the curves are pairwise disjoint,
	\item $\Sigma - \gamma_1 - \dots - \gamma_g$ is connected,
	\item each $\gamma_i$ bounds a disk in $H$.
\end{enumerate}
\end{definition}

\begin{exercise}
Show that $\Sigma - \gamma_1 - \dots - \gamma_g$ is connected if and only if $[\gamma_1], \dots, [\gamma_g]$ are linearly independent in $H_1(\Sigma; \Z)$.
\end{exercise}

See Figure \ref{fig:attachingcircles} for an example of a set of attaching circles.

\begin{figure}[H]
\centering
\labellist
\endlabellist
\includegraphics[scale=0.8]{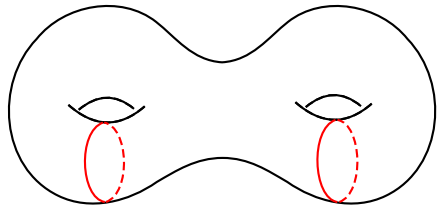}
\vspace{8pt}\caption{A set of attaching circles for the ``obvious'' handlebody in $\R^3$ bounded by $\Sigma$.}
\label{fig:attachingcircles}
\end{figure}

\begin{definition}
A \emph{Heegaard diagram} compatible with $Y = H_1 \cup_f H_2$ is a triple $\cH = (\Sigma, \bfalpha, \bfbeta)$ such that 
\begin{enumerate}
	\item $\Sigma$ is closed oriented surface of genus $g$,
	\item $\bfalpha = \{ \alpha_1, \dots, \alpha_g \}$ is a set of attaching circles for $H_1$,
	\item $\bfbeta = \{ \beta_1, \dots, \beta_g \}$ is a set of attaching circles for $H_2$.
\end{enumerate}
We call $(\Sigma, \bfalpha, \bfbeta)$ a \emph{Heegaard diagram for $Y$}.
\end{definition}

\begin{remark}
The convention in the field is to draw $\alpha$-circles in red and $\beta$-circles in blue.
\end{remark}

See Figure \ref{fig:HDRP3} for examples of Heegaard diagrams for $\R P^3$.  (Another example of a Heegaard diagram is given in Figure \ref{fig:trefoilsurgery} below.)

\begin{figure}[H]
\vspace{10pt}
\centering
\labellist
	\pinlabel {$\alpha$} at 50 -5
	\pinlabel {$\beta$} at 48 102
	\pinlabel {$\alpha_1$} at 180 -5
	\pinlabel {$\alpha_2$} at 290 -5
	\pinlabel {$\beta_1$} at 168 102
	\pinlabel {$\beta_2$} at 290 75
\endlabellist
\includegraphics[scale=1]{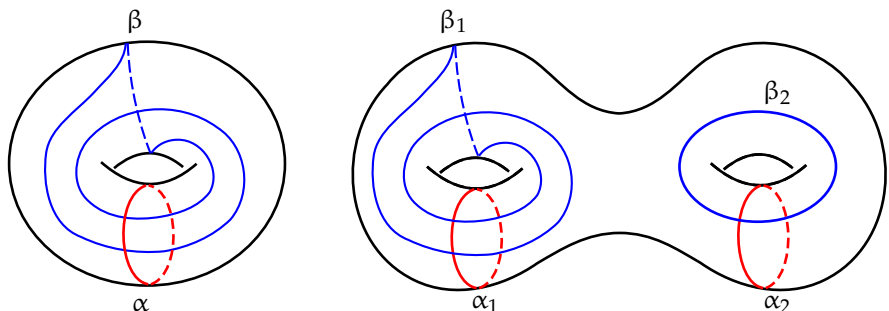}
\vspace{8pt}\caption{Left, a Heegaard diagram for $\R P^3$. Right, a stabilization.}
\label{fig:HDRP3}
\end{figure}

Given a Heegaard diagram $(\Sigma, \bfalpha, \bfbeta)$ for $Y$, we can build a 3-manifold as follows. Thicken $\Sigma$ to $\Sigma \times [0,1]$. Attach thickened disks along $\alpha_i \times \{0\}$, $1 \leq i \leq g$, and along $\beta_i \times \{ 1 \}$, $1 \leq i \leq g$. 

\begin{exercise}
Verify that since $\bfalpha$ and $\bfbeta$ are sets of attaching circles, the boundary of the resulting 3-manifold is homeomorphic to $S^2 \sqcup S^2$. 
\end{exercise}

\noindent Now fill in each of these boundary components with a copy of $B^3$; there is a unique way to do so, since any orientation preserving homeomorphism from $S^2$ to itself is isotopic to the identity. The resulting 3-manifold is homeomorphic to $Y$.

\begin{exercise}\label{exer:H1HD}
Show that $H_1(Y, \Z) \cong H_1(\Sigma; \Z)/ \langle [\alpha_1], \dots [\alpha_g], [\beta_1], \dots, [\beta_g] \rangle$.
\end{exercise}

\begin{exercise}\label{ex:presentationmatrix}
Let $\cH = (\Sigma, \bfalpha, \bfbeta)$ be a Heegaard diagram for $Y$. Compute $H_1(Y; \Z)$ from $\cH$ as follows. Choose an order and orientation on the $\alpha$- and $\beta$-circles and form the matrix $M = (M_{ij})$ where $M_{ij}$ is the algebraic intersection number between the $i^\textup{th}$ $\alpha$-circle and the $j^\textup{th}$ $\beta$-circle. Show that $M$ is a presentation matrix for $H_1(Y; \Z)$.
\end{exercise}

We now consider three ways to alter a Heegaard diagram that do not change the associated 3-manifold. These moves are called \emph{Heegaard moves} and consist of isotopies, handleslides, and stabilizations/destabilizations. Isotopies and handleslides do not change the genus of the Heegaard diagram, while stabilizations (respectively destablizations) increase the genus by one (respectively decrease the genus by one).

Let $\{ \gamma_1, \dots, \gamma_g\}$ be a set of attaching circles for a handlebody $H$ where $\d H = \Sigma$. The set $\{ \gamma_1, \dots, \gamma_g \}$ is \emph{isotopic} to $\{ \gamma'_1, \dots, \gamma'_g \}$ if there is a 1-parameter family of disjoint simple closed curves starting at $\{ \gamma_1, \dots, \gamma_g \}$ and ending at $\{ \gamma'_1, \dots, \gamma'_g \}$.

A \emph{handleslide} of, say,  $\gamma_1$ over $\gamma_2$ produces a new set of attaching circles $\{ \gamma'_1, \gamma_2, \dots, \gamma_g \}$ where $\gamma'_1$ is any simple closed curve disjoint from $\gamma_1, \dots, \gamma_g$
such that $\gamma'_1, \gamma_1,$ and $\gamma_2$ cobound an embedded pair of pants in $\Sigma - \gamma_3 - \dots - \gamma_g$. See Figure \ref{fig:handleslide}.

\begin{figure}[H]
\centering
\labellist
	\pinlabel {$\gamma'_1$} at 140 50
	\pinlabel {$\gamma_1$} at 98 62
	\pinlabel {$\gamma_2$} at 188 62
\endlabellist
\includegraphics[scale=1]{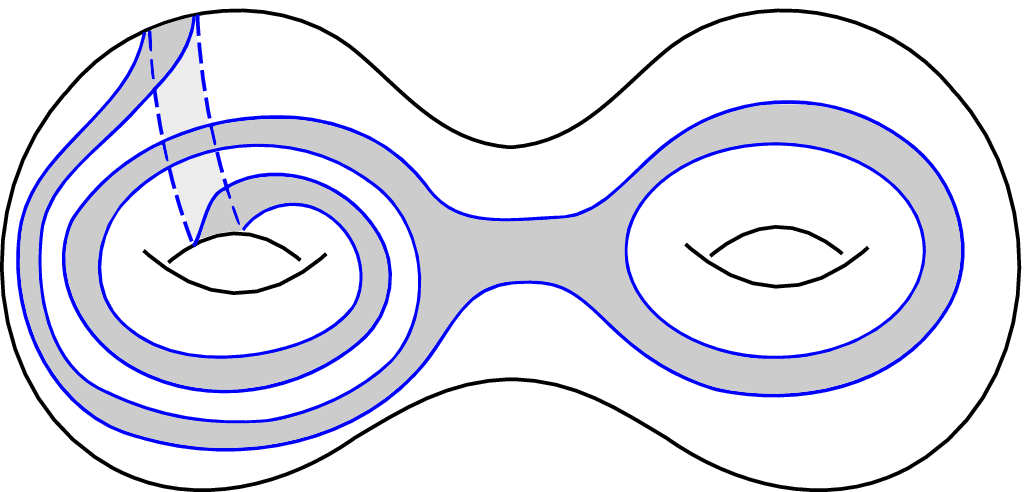}
\vspace{8pt}\caption{A handleslide.}
\label{fig:handleslide}
\end{figure}

Here is another way to think of a handleslide. Suppose that $\gamma_1$ and $\gamma_2$ can be connected by an arc $\delta$ in $\Sigma - \gamma_3 - \dots - \gamma_g$. Let $\gamma'_1$ be the connected sum of  $\gamma_1$ with a parallel copy of $\gamma_2$, where the connected sum is taken along a neigborhood of $\delta$. See Figure \ref{fig:handleslide2}. 

\begin{exercise}
Prove that these two descriptions of handleslides agree (up to isotopy).
\end{exercise}

\begin{figure}[H]
\centering
\labellist
	\pinlabel {$\delta$} at 102 38
	\pinlabel {$\gamma_1$} at 85 58
	\pinlabel {$\gamma_2$} at 139 58
	\pinlabel {$\gamma'_1$} at 305 58
	\pinlabel {$\gamma_2$} at 358 58
\endlabellist
\includegraphics[scale=0.9]{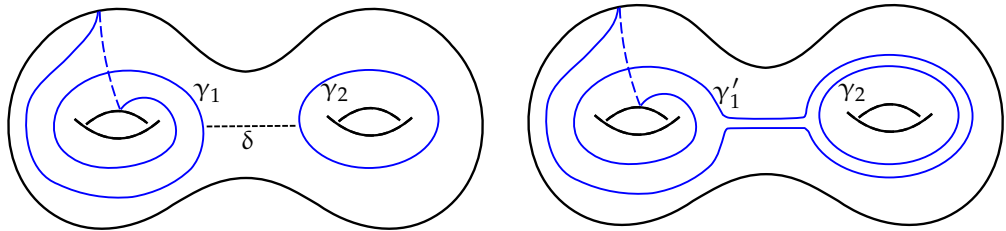}
\vspace{8pt}\caption{Another way to view a handleslide.}
\label{fig:handleslide2}
\end{figure} 

A \emph{stabilization} of a Heegaard diagram $(\Sigma, \bfalpha, \bfbeta)$ results in a Heegaard diagram $(\Sigma', \bfalpha', \bfbeta')$ where
\begin{enumerate}
	\item $\Sigma' = \Sigma \# T^2$, where $T^2 = S^1 \times S^1$,
	\item $\bfalpha' = \bfalpha \cup \{ \alpha_{g+1} \}$ and $\bfbeta' = \bfbeta \cup \{ \beta_{g+1} \}$, where $\alpha_{g+1}$ and $\beta_{g+1}$ are two simple closed curves supported in $T^2$ intersecting transversally in a single point.
\end{enumerate}
We say that $(\Sigma, \bfalpha, \bfbeta)$ is a \emph{destablilization} of $(\Sigma', \bfalpha', \bfbeta')$.

\begin{exercise}
Show that a stabilization of a Heegaard diagram corresponds to a stabilization of the corresponding Heegaard splitting.
\end{exercise}

We have the following standard fact about Heegaard diagrams (see, for example, \cite[Proposition 2.2]{OS3manifolds1}).

\begin{theorem}\label{thm:Heegaardmoves}
Let $(\Sigma, \bfalpha, \bfbeta)$ and $(\Sigma', \bfalpha', \bfbeta')$ be two Heegaard diagrams for $Y$. Then after applying a finite sequence of isotopies, handleslides, and stabilizations to each of them, the two diagrams become homeomorphic (i.e., there is a homeomorphism $\Sigma \to \Sigma'$ taking $\bfalpha$ to $\bfalpha'$ and $\bfbeta$ to $\bfbeta'$, setwise).
\end{theorem}

We will be interested in \emph{pointed Heegaard diagrams}, that is, tuples $(\Sigma, \bfalpha, \bfbeta, w)$ where $w$ is a basepoint in $\Sigma - \bfalpha - \bfbeta$. We now consider pointed isotopies, where the isotopies are not allowed to pass over $w$, and pointed handleslides, where $w$ is not allowed to be in the pair of pants involved in the handleslide. We have the following upgraded version of Theorem \ref{thm:Heegaardmoves}.

\begin{theorem}[{\cite[Proposition 7.1]{OS3manifolds1}}]\label{thm:pointedHeegaardmoves}
Let $(\Sigma, \bfalpha, \bfbeta, w)$ and $(\Sigma', \bfalpha', \bfbeta', w')$ be two pointed Heegaard diagrams for $Y$. Then after applying a finite sequence of pointed isotopies, pointed handleslides, and stabilizations to each of them, the two diagrams become homeomorphic.
\end{theorem}

\subsection{Doubly pointed Heegaard diagrams}
We will also be interested in describing knots inside our 3-manifolds. For simplicity, we will focus on the case where the ambient 3-manifold is $S^3$.

\begin{definition}
A \emph{doubly pointed Heegaard diagram} for a knot $K \subset S^3$ is a tuple $(\Sigma, \bfalpha, \bfbeta, w, z)$, where $w, z$ are basepoints in $\Sigma - \bfalpha - \bfbeta$, such that 
\begin{enumerate}
	\item $(\Sigma, \bfalpha, \bfbeta)$ is a Heegaard diagram for $S^3$, 
	\item $K$ is the union of arcs $a$ and $b$ where $a$ is an arc in $\Sigma - \bfalpha$ connecting $w$ to $z$, pushed slightly into $H_1$ and $b$ is an arc in $\Sigma - \bfbeta$ connecting $z$ to $w$, pushed slightly into $H_2$.
\end{enumerate}
\end{definition}

See Figure \ref{fig:HDtrefoil} for an example of a doubly pointed Heegaard diagram for the left-handed trefoil.

\begin{figure}%[H]
\centering
\labellist
\pinlabel {$\scriptstyle\bullet$} at 76 35
\pinlabel {$w$} at 74 29
\pinlabel {$\scriptstyle\bullet$} at 59 16
\pinlabel {$z$} at 62 12
\pinlabel {\small{$a$}} at 64 37
\pinlabel {\small{$b$}} at 71 18
\pinlabel {\small{$c$}} at 75 3
%\pinlabel {$w$} at 62 34
%\pinlabel {$z$} at 77 16
%\pinlabel {\small{$a$}} at 74 38
%\pinlabel {\small{$b$}} at 65 20
%\pinlabel {\small{$c$}} at 75 5
\endlabellist
\includegraphics[scale=1]{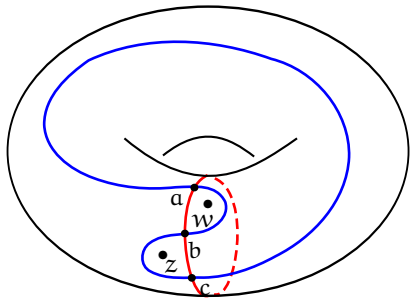}
\vspace{8pt}\caption{A doubly pointed Heegaard diagram for the left-handed trefoil, $-T_{2,3}$. (The labelled points $a, b,$ and $c$ will be used in Section \ref{sec:HFK}.)}
\label{fig:HDtrefoil}
\end{figure}

Given a knot diagram $D$ for a knot $K$, one can obtain a doubly pointed Heegaard diagram for $K$ as follows. Suppose that $D$ has $c$ crossings. Forgetting the crossing data of the diagram $D$ yields an immersed curve $C$ in the plane. The complement of $C$ is $c+2$ regions in the plane, one of which is unbounded. Let $\Sigma$ be the boundary of a regular neighborhood of $C$ in $\R^3$; note that $\Sigma$ is a surface of genus $c+1$. For each of the bounded regions in the complement of $C$, we put a $\beta$-circle on $\Sigma$. For each crossing of $D$, we put an $\alpha$-circle on $\Sigma$ as in Figure \ref{fig:crossing}. Lastly, we add an $\alpha$-circle, say $\alpha_{c+1}$, corresponding to a meridan of $K$. Place a $w$-basepoint on one side of $\alpha_{c+1}$ and a $z$-basepoint on the other side. (Note that which side one choose for $w$ determines the orientation of $K$.) See Figure \ref{fig:trefoilgenus4} for an example.

\begin{figure}%[H]
\centering
\labellist
\endlabellist
\includegraphics[scale=1]{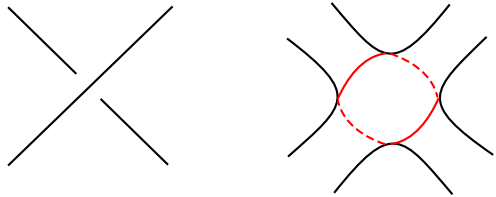}
\vspace{8pt}\caption{A knot crossing and the corresponding portion of the associated doubly pointed Heegaard diagram.}
\label{fig:crossing}
\end{figure}

\begin{figure}%[H]
\centering
\labellist
\pinlabel {$\scriptstyle\bullet$} at 250 162
\pinlabel {$w$} at 248 156
\pinlabel {$\scriptstyle\bullet$} at 273 162
\pinlabel {$z$} at 275 156
\endlabellist
\includegraphics[scale=1]{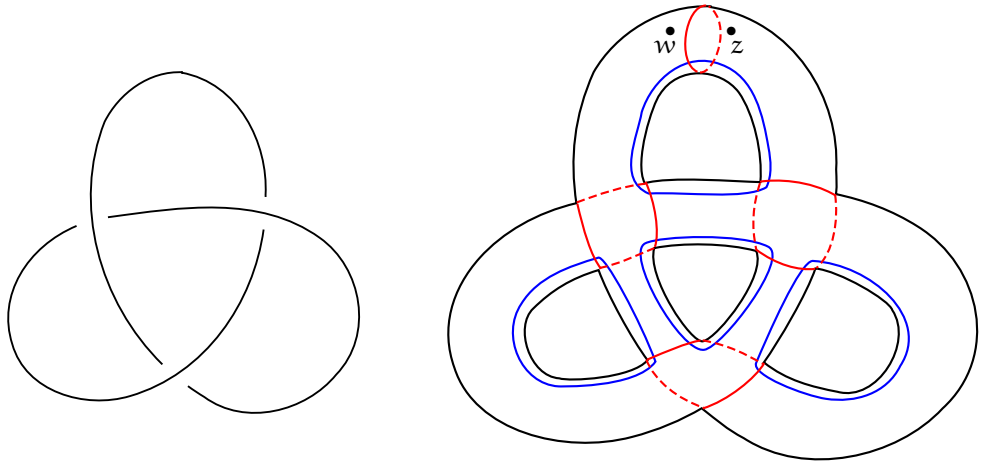}
\vspace{8pt}\caption{Another doubly pointed Heegaard diagram for the left-handed trefoil, $-T_{2,3}$.}
\label{fig:trefoilgenus4}
\end{figure}

\begin{figure}%[H]
\centering
\includegraphics[scale=1]{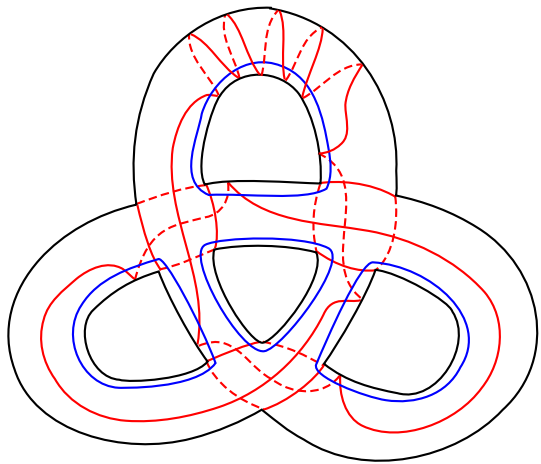}
\vspace{8pt}\caption{A Heegaard diagram for $S^3_{+5}(-T_{2,3})$.}
\label{fig:trefoilsurgery}
\end{figure}

\begin{exercise}
Show that the above construction yields a doubly pointed Heegaard diagram for $K \subset S^3$.
\end{exercise}

In the aforementioned construction, if we replace the circle $\alpha_{c+1}$ with an $n$-framed longitude and remove the basepoint $z$, we obtain a pointed Heegaard diagram for $S^3_n(K)$. See Figure \ref{fig:trefoilsurgery} for an example.

We now consider doubly pointed isotopies, which are required to miss both $w$ and $z$, and doubly pointed handleslides, where neither $w$ nor $z$ is allowed to be in the pair of pants involved in the handleslide. We have the following analog of Theorem \ref{thm:pointedHeegaardmoves}.

\begin{theorem}[{\cite[Proposition 3.5]{OSknots}}]
Let $(\Sigma, \bfalpha, \bfbeta, w, z)$ and $(\Sigma', \bfalpha', \bfbeta', w', z')$ be two doubly pointed Heegaard diagrams for $K \subset S^3$. Then after applying a finite sequence of doubly pointed isotopies, doubly pointed handleslides, and stabilizations to each of them, the two diagrams become homeomorphic.
\end{theorem}

\begin{exercise}
Find a sequence of doubly pointed Heegaard moves from Figure \ref{fig:HDtrefoil} to Figure \ref{fig:trefoilgenus4}.
\end{exercise}

In Sections \ref{sec:HF} and \ref{sec:HFK}, we will define invariants of 3-manifolds and knots in $S^3$. These invariants will be defined in terms of pointed and doubly pointed Heegaard diagrams, and invariance will follow from the fact that the invariants remain unchanged under pointed and doubly pointed Heegaard moves.

We conclude this section with some additional exercises.

\begin{exercise}\label{exer:41}
Find a genus 1 doubly pointed Heegaard diagram for the figure eight knot.
\end{exercise}

\begin{exercise}\label{exer:Tpq}
Find a genus 1 doubly pointed Heegaard diagram for the torus knot $T_{p,q}$.
\end{exercise}

\begin{exercise}\label{exer:diagramchanges}
Let $(\Sigma, \bfalpha, \bfbeta)$ be a Heegaard diagram for $Y$. What is the manifold described by $(-\Sigma, \bfalpha, \bfbeta)$? By $(\Sigma, \bfbeta, \bfalpha)$?
\end{exercise}

\begin{exercise}
Let $(\Sigma, \bfalpha, \bfbeta, w, z)$ be a Heegaard diagram for a knot $K$ in $S^3$. What is the knot described by $(\Sigma, \bfalpha, \bfbeta, z, w)$? By $(-\Sigma, \bfbeta, \bfalpha, z, w)$?
\end{exercise}

%%%%%%%%%%%%%%%%%%%%%%%%%%%%%%
%%%%%%%%%%%%%%%%%%%%%%%%%%%%%%

\section{Heegaard Floer homology}\label{sec:HF}
\subsection{Overview}
From a Heegaard diagram $\cH$ for $Y$, we will build chain complexes $\CFhat(\cH)$ and $\CFm(\cH)$ whose chain homotopy types are invariants of $Y$; the former is a finitely generated chain complex over $\F$ and the latter is a finitely generated graded chain complex over $\F[U]$. Throughout, $\F$ denotes the field $\Z/2\Z$ and $U$ is a formal variable of degree $-2$.  We denote the homology of these complexes by $\HFhat(Y)$ and $\HFm(Y)$ respectively; the former is a graded vector space over $\F$ and the latter is a graded module over $\F[U]$. When discussing properties that apply to either flavor of Heegaard Floer homology, will write $\HF^\circ$ rather than $\HFhat$, $\HFm$, $\HFp$, or $\HFinf$ (the latter two of which are defined below). The material in this section draws from \cite[Sections 4-8]{OSintro} and \cite[Sections 3 and 4]{OS3manifolds1}.

For simplicity, we will consider the case where $Y$ is a rational homology sphere. The case $b_1(Y) > 0$ requires some additional admissibility assumptions on $\cH$; see \cite[Section 4.2]{OS3manifolds1}. Before defining these invariants, we discuss certain aspects of their formal structure.

The Heegaard Floer homology of $Y$ splits as a direct sum over $\spinc$ structures on $Y$:
\[ \HF^\circ(Y) = \bigoplus_{\mfs \in \spinc(Y)} \HF^\circ(Y, \mfs) \]
See \cite[Section 6]{OSintro} for a discussion of $\spinc$ structures in terms of homotopy classes of non-vanishing vector fields. Note that $\spinc$ structures on $Y$ are in (non-canonical) bijection with elements in $H^2(Y; \Z) \cong H_1(Y; \Z)$. The above splitting on the level of homology comes from a splitting on the chain level.

Both $\HFhat(Y)$ and $\HFm(Y)$ are finitely generated and graded. Finitely generated graded vector spaces are simply a direct sum of graded copies of $\F$. Finitely generated graded modules over a PID are also completely characterized. (See, for example, \cite[Proposition A.4.3]{OSSgrid}.) Since the only homogenously graded polynomials in $\F[U]$ are the monomials $U^n$, any finitely generated graded module over $\F[U]$ is isomorphic to 
\begin{equation}\label{eq:fingenmod}
	\bigoplus_i \F[U]_{(d_i)} \oplus \bigoplus_j \F[U]_{(c_j)}/(U^{n_j}),
\end{equation}
where $\F[U]_{(d)}$ denotes the ring $\F[U]$ where the element $1$ has grading $d$.

Moreover, by \cite[Theorem 10.1]{OS3manifolds2}, we have that for a rational homology sphere $Y$, for all $\mfs \in \spinc(Y)$,
\begin{equation}\label{eq:HFm1summand}
\HFm(Y, \mfs) \cong \F[U]_{(d)} \oplus \bigoplus_j \F[U]_{(c_j)}/U^{n_j},
\end{equation}
that is, there is exactly one free summand in $\HFm(Y, \mfs)$. We define the \emph{$d$-invariant} of $(Y, \mfs)$ to be
\[ d(Y, \mfs) = \max \{ \gr(x) \mid x \in \HFm(Y, \mfs),\ U^N x \neq 0 \ \forall \ N > 0 \}. \]
Note that $d(Y, \mfs)$ is equal to $d$ in Equation \eqref{eq:HFm1summand}.
The $U$-torsion part is called $\HFred(Y, \mfs)$; that is, using the notation in Equation \eqref{eq:HFm1summand}, 
\[ \HFred(Y, \mfs) = \bigoplus_j \F[U]_{(c_j)}/U^{n_j}. \]
A rational homology sphere $Y$ with $\HFred(Y, \mfs) = 0$ for all $\mfs \in \spinc(Y)$ is called an \emph{L-space}.

\begin{remark}\label{rk:grconv}
Different grading conventions exist in the literature. We have chosen our grading convention so that $\HFm(S^3) \cong \F[U]_{(0)}$, as opposed to the perhaps more common $\F[U]_{(-2)}$. Our grading convention choice simplifies certain formulas, such as the K\"unneth formula \cite[Theorem 1.5]{OS3manifolds2}:
\[ \CFm(Y_1 \# Y_2, \mfs_1 \# \mfs_2) \simeq \CFm(Y_1, \mfs_1) \otimes_{\F[U]}  \CFm(Y_2, \mfs_2). \] 
Note that our choice of grading convention also impacts the gradings on $\HFred(Y, \mfs)$.
\end{remark}

The chain complexes $\CFhat(\cH, \mfs)$ and $\CFm(\cH, \mfs)$ fit in the following $U$-equivariant short exact sequence:
\begin{equation}\label{eqn:minushatses}
0 \to \CFm(\cH, \mfs) \xrightarrow{\cdot U} \CFm(\cH, \mfs) \to \CFhat(\cH, \mfs) \to 0,
\end{equation}
yielding the $U$-equivariant exact triangle:
\begin{equation*}
\label{pic:exact}
\begin{tikzpicture}[baseline=(current  bounding  box.center)]
\node(1)at(0,0){$\HFhat(Y, \mfs)$.};
\node(2)at (-2,1){$\HFm(Y, \mfs)$};
\node(3)at (2,1){$\HFm(Y, \mfs)$};
\path[->](2)edge node[above]{$\cdot U$}(3);
\path[->](3)edge (1);
\path[->](1)edge(2);
\end{tikzpicture}
\end{equation*}

\begin{remark}\label{rk:Lspace}
Equation \eqref{eq:HFm1summand} and the above exact triangle imply that for a rational homology sphere $Y$, we have $\dim \HFhat(Y) \geq |H_1(Y; \Z)|$; cf. Exercise \ref{exer:EulerHF}. It follows that $\dim \HFhat(Y) = |H_1(Y; \Z)|$ if and only if $Y$ is an L-space.
\end{remark}

We may also consider $\HFinf(Y, \mfs) = H_*(\CFinf(Y, \mfs))$ where 
\[ \CFinf(Y, \mfs) = \CFm(Y, \mfs) \otimes_{\F[U]} \F[U, U^{-1}]. \]
Note that $\CFm(Y, \mfs) \subset \CFinf(Y, \mfs)$ and define $\CFp(Y, \mfs)$ to be the quotient 
\[ \CFinf(Y, \mfs)/CF^-(Y, \mfs). \] 
We have the following short exact sequence:
\[ 0 \to \CFm(Y, \mfs) \to \CFinf(Y, \mfs) \to \CFp(Y, \mfs) \to 0.\]
The above short exact sequence on the chain level induces the following $U$-equivariant exact triangle:
\begin{equation}
\label{eq:exact}
\begin{tikzpicture}[baseline=(current  bounding  box.center)]
\node(1)at(0,0){$\HFp(Y, \mfs)$.};
\node(2)at (-2,1){$\HFm(Y, \mfs)$};
\node(3)at (2,1){$\HFinf(Y, \mfs)$};
\path[->](2)edge node[above]{}(3);
\path[->](3)edge (1);
\path[->](1)edge(2);
\end{tikzpicture}
\end{equation}

\begin{exercise}\label{exer:HFred}
Use Equation \eqref{eq:HFm1summand} to show that when $Y$ is a rational homology sphere, $\HFinf(Y, \mfs) \cong \F[U, U^{-1}]$. Use this fact combined with Equation \eqref{eq:exact} to show that if 
\[ \HFm(Y, \mfs) \cong \F[U]_{(d)} \oplus \bigoplus_j \F[U]_{(c_j)}/U^{n_j}, \]
then 
\[ \HFp(Y, \mfs) \cong \mathcal{T}^+_{(d+2)} \oplus \bigoplus_j  \F[U]_{(c_j+1)}/U^{n_j}, \]
where $\mathcal{T}^+ = \F[U, U^{-1}]/U \F[U]$.
%Use this fact to give a description of $\HFp(Y, \mfs)$ in terms of $d(Y, \mfs)$ and $\HFred(Y, s)$.
\end{exercise}

\begin{remark}
In light of Exercise \ref{exer:HFred}, one may define $\HFred(Y, \mfs)$ in terms of $\HFm(Y, \mfs)$ or $\HFp(Y, \mfs)$; this choice also affects the grading of $\HFred(Y, \mfs)$. The most common convention in the literature is to define $\HFred$ in terms of $\HFp$.
\end{remark}

A smooth cobordism $W$ from $Y_0$ to $Y_1$ (that is, a compact smooth 4-manifold $W$ with $\d W = -Y_0 \sqcup Y_1$) induces a map from the Heegaard Floer homology of $Y_0$ to the Heegaard Floer homology of $Y_1$. More specifically, given a $\spinc$ structure $\mft$ on $W$, we have a homomorphism
\[ F^\circ_{W, \mft} \co \HF^\circ(Y_0, \mft|_{Y_0}) \to \HF^\circ(Y_1, \mft|_{Y_1}), \]
which is defined via a handle decomposition of $W$ (but does not depend on the choice of handle decomposition). The cobordism map $F^\circ_{W, \mft}$ has a grading shift depending only on $W$ and $\mft$. See \cite{OS:four} for further details or \cite[Section 3.2]{OSsurvey2} for an expository overview.

Let $Z$ be a compact 3-manifold with torus boundary, and $\gamma_0, \gamma_1$, and $\gamma_\infty$ three simple closed curves in $\d Z$ such that
\[ \# (\gamma_0 \cap \gamma_1) = \# (\gamma_1 \cap \gamma_\infty) = \# (\gamma_\infty \cap \gamma_0) = -1.\]
See Figure \ref{fig:triple} for an example of such a triple of curves (which justifies the choice of subscripts on $\gamma$). Let $Y_i$ be the result of Dehn filling $Z$ along $\gamma_i$ for $i=0,1, \infty$; that is, $Y_i$ is the union of $Z$ and a solid torus, where $\gamma_i$ is a meridian of the solid torus. Then by \cite[Theorem 9.12]{OS3manifolds2}, we have an exact triangle
\begin{equation}\label{eq:surgeryexacttriangle}
\begin{tikzpicture}[baseline=(current  bounding  box.center)]
\node(1)at(0,0){$\HFhat(Y_\infty)$};
\node(2)at (-2,1){$\HFhat(Y_0)$};
\node(3)at (2,1){$\HFhat(Y_1)$};
\path[->](2)edge node[above]{$\widehat{F}_0$}(3);
\path[->](3)edge node[below right]{$\widehat{F}_1$}(1);
\path[->](1)edgenode[below left]{$\widehat{F}_\infty$}(2);
\end{tikzpicture}
\end{equation}
where $\widehat{F}_i$ is the cobordism map associated to the corresponding 2-handle cobordism. The analogous exact triangle also holds for $\HFp$. 

\begin{figure}[H]
\centering
\labellist
\pinlabel {$\gamma_0$} at 78 50
\pinlabel {$\gamma_1$} at 77 89
\pinlabel {$\gamma_\infty$} at 118 80
\endlabellist
\includegraphics[scale=0.8]{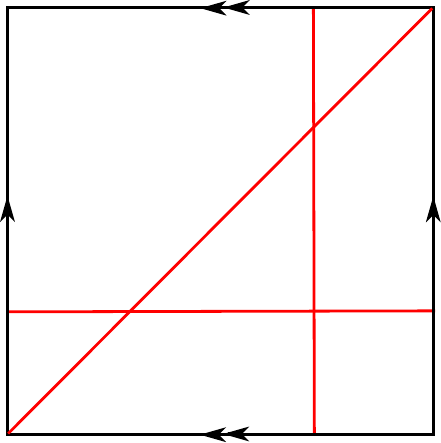}
\vspace{8pt}\caption{}
\label{fig:triple}
\end{figure}

\begin{remark}\label{rk:completions}
For the analogous exact triangle for the minus or infinity flavors, one must work over the formal power series ring $\F[[U]]$ and semi-infinite Laurent polynomials $\F[[U, U^{-1}]$ respectively; see \cite[Section 2]{MOlink}.
\end{remark}

\subsection{The Heegaard Floer chain complex}
Now that we have completed our overview of the structural properties of Heegaard Floer homology, we turn to the construction of the Heegaard Floer chain complex associated to a Heegaard diagram.

Let $\cH = (\Sigma, \bfalpha, \bfbeta, w)$ be a pointed Heegaard diagram for $Y$ where $\Sigma$ has genus $g$, and as usual $\bfalpha = \{ \alpha_1, \dots, \alpha_g \}$ and $\bfbeta = \{ \beta_1, \dots, \beta_g \}$. We further require that the $\alpha$- and $\beta$-circles intersect transversally. 

Consider the $g$-fold symmetric product 
\[ \Sym^g(\Sigma) = \Sigma^{\times g} /S_g, \]
where $S_g$ denotes the symmetric group on $g$-elements. Points in $\Sym^g(\Sigma)$ consist of unordered $g$-tuples of points in $\Sigma$. 

\begin{exercise}
Even though the action of $S_g$ on $\Sigma^{\times g}$ is not free, show that the quotient $\Sym^g(\Sigma)$ is a smooth manifold. (Hint: Fix a complex structure on $\Sigma$. Then use the Fundamental Theorem of Algebra to define a map between ordered and unordered $g$-tuples of complex numbers.)
\end{exercise}

We have two half-dimensional subspaces of $\Sym^g(\Sigma)$:
\[ \bbT_\alpha = \alpha_1 \times \dots \times \alpha_g \qquad \textup{ and } \qquad \bbT_\beta = \beta_1 \times \dots \times \beta_g. \]
The chain complex $\CFhat(\cH)$ is freely generated over $\F$ by $\bbT_\alpha \cap \bbT_\beta$, that is, intersection points between $\bbT_\alpha$ and $\bbT_\beta$. Note that points in $\bbT_\alpha \cap \bbT_\beta$ can be viewed in $\Sigma$ as unordered $g$-tuples of intersection points between the $\alpha$- and $\beta$-circles such that each $\alpha$-circle, respectively $\beta$-circle, is used exactly once. 

\begin{exercise}
Find all of the generators $\bbT_\alpha \cap \bbT_\beta$ in Figure \ref{fig:trefoilsurgery}, viewed as unordered $g$-tuples of points in $\Sigma$.
\end{exercise}

\noindent We will also be interested in the following subspace of $\Sym^g(\Sigma)$:
\[ V_w = \{w \} \times \Sym^{g-1}(\Sigma). \]
This subspace can be viewed as unordered $g$-tuples of points in $\Sigma$ such that at least one point is $w$.

The differential $\d \co \CFhat(\cH) \to \CFhat(\cH)$ will count certain holomorphic disks in $\Sym^g(\Sigma)$. Let $\bbD$ denote the unit disk in $\bbC$, and let $e_\alpha$, respectively $e_\beta$, denote the arc in $\d \bbD$ with $\operatorname{Re} (z) \geq 0$, respectively $\operatorname{Re} (z) \leq 0$. A \emph{Whitney disk} from $\bfx$ to $\bfy$ is a continuous map $\phi \co \bbD \to \Sym^g(\Sigma)$ such that 
\begin{enumerate}
	\item $\phi(-i) = \bfx$,
	\item $\phi(i) = \bfy$,	
	\item $\phi(e_\alpha) \subset \bbT_{\alpha}$,
	\item $\phi(e_\beta) \subset \bbT_{\beta}$.
\end{enumerate}
See Figure \ref{fig:disk}. Let $\pi_2(\bfx, \bfy)$ denote the set of homotopy classes of Whitney disks from $\bfx$ to $\bfy$.

\begin{figure}[H]
\centering
\labellist
\pinlabel {$\bfx$} at 25 -5
\pinlabel {$\bfy$} at 25 53
\endlabellist
\includegraphics[scale=1.5]{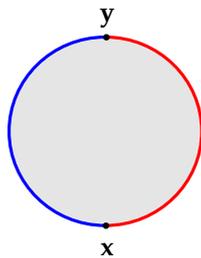}
\vspace{8pt}\caption{A schematic of a Whitney disk.}
\label{fig:disk}
\end{figure}

\begin{exercise}\label{exer:spincepsilon}
Following \cite[Section 5.1]{OSintro}, there is a rather straightforward obstruction to the existence of a Whitney disk from $\bfx$ to $\bfy$. 
\begin{enumerate}
	\item Viewing $\bfx$, respectively $\bfy$, as an unordered $g$-tuple $\{x_1, \dots, x_g \}$, respectively $\{ y_1, \dots, y_g \}$, of points in $\bfalpha \cap \bfbeta$, we may choose a collection of arcs $a \subset \bfalpha$ such that $\d a = y_1 + \dots + y_g - x_1 - \dots - x_g$ and a collection of arcs $b \subset \bfbeta$ such that $\d b = x_1 + \dots + x_g - y_1 - \dots - y_g$. Then $a-b$ is a 1-cycle in $\Sigma$. Using Exercise \ref{exer:H1HD}, verify that the image of $\epsilon(\bfx, \bfy) = [a-b] \in H_1(Y; \Z)$ is well-defined.
	\item Show that $H_1(\Sym^g(\Sigma); \Z) \cong H_1(\Sigma; \Z)$. (Hint: See \cite[Lemma 2.6]{OS3manifolds1}.) Combined with Exercise \ref{exer:H1HD}, conclude that
\[ \frac{H_1(\Sym^g(\Sigma); \Z)}{H_1(\bbT_\alpha; \Z) \oplus H_1(\bbT_\beta; \Z)} \cong H_1(Y; \Z), \]
and that if $\epsilon(\bfx, \bfy) \neq 0 \in H_1(Y; \Z)$, then there cannot exist a Whitney disk in $\Sym^g(\Sigma)$ from $\bfx$ to $\bfy$.
\end{enumerate}
\end{exercise}

A choice of complex structure on $\Sigma$ induces one on $\Sym^g(\Sigma)$. Given $\phi \in \pi_2(\bfx, \bfy)$, let $\cM(\phi)$ denote the moduli space of holomorphic representatives of $\phi$. (\cite[Proposition 3.9]{OS3manifolds1} ensures that under suitably generic conditions, $\cM(\phi)$ is smooth.) The expected dimension of $\cM(\phi)$ is called the \emph{Maslov index} of $\phi$ and denoted $\mu(\phi)$. There is an $\R$-action on $\cM(\phi)$ coming from complex automorphisms of $\bbD$ that preserve $i$ and $-i$. (One way to see this $\R$-action is to use the Riemann mapping theorem to change the unit disk to the infinite strip $[0,1] \times i\R \subset \bbC$ where $e_\alpha$ corresponds to $\{1\} \times i\R$ and $e_\beta$ to $\{0\} \times i\R$. Then the $\R$-action corresponds to vertical translations.) Let $\cMhat(\phi) = \cM(\phi)/\R$. If $\mu(\phi) = 1$, then $\cMhat(\phi)$ is a compact zero dimensional manifold \cite[Theorem 3.18]{OS3manifolds1}.  Let $n_w(\phi)$ denote the algebraic intersection between $\phi(\bbD)$ and $V_w$.

We define a relative $\Z$-grading, called the \emph{Maslov grading}, on $\CFhat(\cH)$ as follows. (Recall that a relative grading defines the difference between the gradings of two elements.)
Let $\phi \in \pi_2(\bfx, \bfy)$. Define 
\[ \gr(\bfx) - \gr(\bfy) = \mu(\phi) - 2n_w(\phi) \]
By \cite[Proposition 2.15]{OS3manifolds1}, this relative grading is well-defined whenever $\epsilon(\bfx, \bfy) = 0$, where $\epsilon$ is the function defined in Exercise \ref{exer:spincepsilon}.

We may use the function $\epsilon$ to partition the intersection points in $\bbT_\alpha \cap \bbT_\beta$. These equivalence classes are in bijection with $H_1(Y; \Z)$ and hence are in bijection with $\spinc(Y)$; see \cite[Section 2.6]{OS3manifolds1} for more details. In particular, we have a splitting $\CFhat(\cH) = \oplus_{\mfs \in \spinc(Y)} \CFhat(\cH, \mfs)$.

The differential $\d \co \CFhat(\cH, \mfs) \to \CFhat(\cH, \mfs)$ is defined to be
\[ \d \bfx = \sum_{\bfy \in \bbT_\alpha \cap \bbT_\beta} \sum_{\substack{\phi \in \pi_2(\bfx, \bfy) \\ \mu(\phi) = 1 \\n_w(\phi) = 0}} \# \cMhat(\phi) \; \bfy. \]
Note that in the first summation, it suffices to only consider $\bfy$ with $\epsilon(\bfx, \bfy) = 0$; that is, the differential respects the splitting $\CFhat(\cH) = \oplus_{\mfs \in \spinc(Y)} \CFhat(\cH, \mfs)$. It follows from the definition of the relative Maslov grading that the differential $\d$ lowers the Maslov grading by one. By \cite[Theorem 4.1]{OS3manifolds1}, $\d^2=0$. Let 
\[ \HFhat(\cH, \mfs) = H_*(\CFhat(\cH, \mfs)). \]

\begin{remark}\label{rk:Z2grading}
We may define a relative $\Z/2\Z$ grading on $\CFhat(\cH)$ that agrees with the mod $2$ reduction of the relative Maslov grading as follows. Following \cite[Section 5]{OS3manifolds2}, choose an orientation of $\bbT_\alpha$ and $\bbT_\beta$ and define the relative $\Z/2\Z$ grading between $\bfx, \bfy \in \bbT_\alpha \cap \bbT_\beta$ to be the product of their local intersection numbers. (Here, we are identifying $\Z/2\Z$ with $\{ \pm 1 \}$ under multiplication.)
\end{remark}

\begin{exercise}\label{exer:EulerHF}
Let $\cH$ be a Heegaard diagam for a rational homology sphere $Y$. Use Remark \ref{rk:Z2grading} to prove that $\chi(\HFhat(\cH)) = \pm |H_1(Y; \Z)|$ and conclude that $\operatorname{dim} \HFhat(\cH) \geq |H_1(Y; \Z)|$. (Hint: Use Exercise \ref{ex:presentationmatrix}.) Compare with Remark \ref{rk:Lspace}.
\end{exercise}

We now define the chain complex $\CFm(\cH)$, which is freely generated over $\F[U]$ by $\bbT_\alpha \cap \bbT_\beta$. Here, $U$ is a formal variable with $\gr(U) = -2$. We no longer require that Whitney disks miss the basepoint (i.e., we remove the $n_w (\phi) = 0$ requirement), and instead use the variable $U$ to count the algebraic intersection number $n_w(\phi)$ of $\phi(\bbD)$ and $V_w$:

\[ \d \bfx = \sum_{\bfy \in \bbT_\alpha \cap \bbT_\beta} \sum_{\substack{\phi \in \pi_2(\bfx, \bfy) \\ \mu(\phi) = 1}} \# \cMhat(\phi) U^{n_w(\phi)} \; \bfy. \]
The definition of $\d$ is extended to all elements of $\CFm(\cH)$ by $\F[U]$-linearity. As in the case of $\CFhat$, the chain complex $\CFm(\cH)$ splits over $\spinc(Y)$, that is, $\CFm(\cH) = \oplus_{\mfs \in \spinc(Y)} \CFm(\cH, \mfs)$.
By \cite[Theorem 4.3]{OS3manifolds1}, $\d^2 =0$. Let 
\[ \HFm(\cH, \mfs) = H_*(\CFm(\cH, \mfs)). \]

\begin{theorem}[{\cite[Theorem 1.1]{OS3manifolds1}}]\label{thm:isotype}
Let $\cH$ be a pointed Heegaard diagram for $Y$. Then the isomorphism type of $\HFhat(\cH, \mfs)$ and  $\HFm(\cH, \mfs)$  is an invariant of $Y$ and $\mfs$.
\end{theorem}

%\begin{remark}
%Recall that the chain homotopy type of a chain complex over a PID (e.g., $\F$, $\F[U]$) is determined by its homology.
%\end{remark}

\noindent
In order to prove the above theorem, one must show that $\HFhat(\cH, \mfs)$ is independent of the choice of Heegaard diagram, basepoint, and complex structure. Indeed, Ozsv\'ath-Szab\'o show that  Heegaard moves induce homotopy equivalences on the associated chain complexes, as do changes in complex structure.

\begin{remark}
Note that $\CFhat(\cH) = \CFm(\cH)/(U=0)$, giving rise to the short exact sequence in \eqref{eqn:minushatses}.
\end{remark}

In \cite[Theorem 7.1]{OS:four}, Ozsv\'ath-Szab\'o prove that the relative $\Z$-grading may be lifted to a well-defined absolute $\Q$-grading; this is done by considering a cobordism from $S^3$ to $Y$ and considering holomorphic triangles associated to a Heegaard triple.

We conclude this section with some computations.

\begin{exercise}\label{exer:Lpq}
Compute $\HFhat(L(p,q), \mfs)$ and $\HFm(L(p,q), \mfs)$ for all $\mfs \in \spinc(L(p,q))$.
\end{exercise}

In general, computing $\HFhat(Y)$ and $\HFm(Y)$ from a Heegaard diagram is not easy. In Section \ref{sec:surgery}, we will see how to compute $\HFhat(Y)$ and $\HFm(Y)$ when $Y$ is surgery on a knot in $S^3$.

\begin{exercise}
Recall that $pq \pm 1$ surgery on the torus knot $T_{p,q}$ is a lens space. Use Exercise \ref{exer:Lpq} and Equation \eqref{eq:surgeryexacttriangle} to compute $\HFhat(S^3_n(T_{p,q}))$ for $n \geq pq-1$.
\end{exercise}

\begin{exercise}\label{exer:conj}
Let $\cH = (\Sigma, \bfalpha, \bfbeta, w)$ be a Heegaard diagram for $Y$. Let $\cH' = (-\Sigma, \bfbeta, \bfalpha, w)$. Show that the chain complex $\CFm(\cH)$ is isomorphic to $\CFm(\cH')$.
\end{exercise}

\begin{exercise}
Use Exercise \ref{exer:diagramchanges} and the fact that $d(-Y, \mfs) = -d(Y, \mfs)$ \cite[Proposition 4.2]{OS3manifolds2} to determine the relationship between $\HFm(Y)$ and $\HFm(-Y)$ as absolutely graded modules.
\end{exercise}

In general, direct computations of Heegaard Floer homology are difficult. Sarkar and Wang \cite{SarkarWang} proved that $\HFhat$ can be computed combinatorially using nice diagrams. Lipshitz-Ozsv\'ath-Thurston \cite{LOTfactoring} give an alternative combinatorial method for computing $\HFhat$ using bordered Floer homology, a version of Heegaard Floer homology for 3-manifolds with boundary. Manolescu-Ozsv\'ath-Thurston \cite{MOT} show that all flavors of Heegaard Floer homology are algorithmically computable in terms of a link surgery description for a 3-manifold and a grid diagram for that link. In Section \ref{sec:surgery}, we will see how to compute $\HFm$ of surgery on a knot in $S^3$ in terms of the knot Floer complex, a knot invariant which we describe in Section \ref{sec:HFK} below.

Lastly, we note that Theorem \ref{thm:isotype} merely states that the isomorphism type of $\HFm(Y)$ is an invariant of $Y$. Juh\'asz-Thurston-Zemke \cite{JuhaszThurstonZemke} prove that Heegaard Floer homology in fact assigns a concrete group (that is, not just a group up to isomorphism) to a 3-manifold $Y$ equipped with a basepoint.

%%%%%%%%%%%%%%%%%%%%%%%%%%%%%%
%%%%%%%%%%%%%%%%%%%%%%%%%%%%%%

\section{Knot Floer homology}\label{sec:HFK}
\subsection{Overview}
Let $K$ be a knot in $S^3$. The simplest version of the knot invariant is $\HFKhat(K)$, a bigraded vector space over the field $\F=\Z/2\Z$:
\[ \HFKhat(K) = \bigoplus_{m,s \in \Z} \HFKhat_m(K,s), \]
where $\HFKhat(K)$ is supported in finitely many bigradings.
Knot Floer homology categorifies the Alexander polynomial \cite[Equation (1)]{OSknots} in the sense that the graded Euler characteristic of $\HFKhat(K)$ is $\Delta_K(t)$:
\[ \Delta_K(t) = \sum_{m, s} (-1)^m \dim \HFKhat_{m}(K,s) \; t^s. \]
Moreover, knot Floer homology strengthens two key properties of the Alexander polynomial. Let 
\[ \Delta_K(t) = a_0 + \sum_{s > 0} a_s(t^s+t^{-s}) \]
denote the symmetrized Alexander polynomial. While the Alexander polynomial gives a lower bound on the genus of $K$ in the following manner:
\[ g(K) \geq \max \{ s \mid a_s \neq 0 \}, \]
knot Floer homology actually detects $g(K)$ \cite{OSgenus}:
\begin{equation}\label{eq:genus}
	g(K) = \max \{ s \mid \HFKhat(K, s) \neq 0\}.
\end{equation}
Similarly, while the Alexander polynomial obstructs fiberedness in that 
\[ K \textup{ is fibered } \Rightarrow \ a_{g(K)} = \pm 1 , \]
knot Floer homology actually detects fiberedness \cite{Ghiggini, Ni}:
\[ K \textup{ is fibered } \Leftrightarrow \HFKhat(K, g(K)) = \F. \]

\subsection{The knot Floer complex}
We now modify the constructions in Section \ref{sec:HF} to the case of doubly pointed Heegaard diagrams in order to define knot invariants. As mentioned above, for simplicity, we restrict ourselves to knots in $S^3$. (With mild modifications, the constructions described here apply to any null-homologous knot in a rational homology sphere.) Some of this material comes from \cite[Section 10]{OSintro}; see \cite{OSknots} for more details and proofs. Many of our conventions and notations come from \cite{Zemkeabsgr}; see especially \cite[Section 1.5]{Zemkeabsgr}. Knot Floer homology was independently defined by Rasmussen in \cite{RasmussenThesis}.

Recall that in the differential of the chain complex $\CFm$, the variable $U$ recorded information about the basepoint $w$. Now that we have two basepoints, we will work with the two variable polynomial ring $\F[U, V]$.
We endow this ring with a bigrading $\gr=(\gr_U, \gr_V)$. We call $\gr_U$ the \emph{$U$-grading} and $\gr_V$ the \emph{$V$-grading}. The variables $U$ and $V$ have grading 
\[ \gr(U)= (-2, 0) \qquad \textup{ and } \qquad \gr(V) = (0, -2). \]
It will often be convenient to consider the following linear combination of $\gr_U$ and $\gr_V$,
\[ A = \frac{1}{2}(\gr_U - \gr_V), \]
called the \emph{Alexander grading}. Note that $A(U) = -1$ and $A(V) = 1$.

Let $\cH$ be a doubly pointed Heegaard diagram for a knot $K \subset S^3$. Let $\CFKR(\cH)$ be the free $\F[U, V]$-module generated by $\bbT_\alpha \cap \bbT_\beta$. This module is relatively bigraded as follows. Let $\phi \in \pi_2(\bfx, \bfy)$ and define
\begin{align*}
	 \gr_U(\bfx) - \gr_U (\bfy) &= \mu(\phi) -2n_w(\phi) \\
	 \gr_V(\bfx) - \gr_V (\bfy) &= \mu(\phi) -2n_z(\phi).
\end{align*}
By \cite[Proposition 7.5]{OS3manifolds2} (see also \cite[Section 5.1]{Zemkeabsgr}), this relative grading is well-defined. The relative gradings $\gr_U$ and $\gr_V$ can be lifted to absolute $\Z$-gradings, using the absolute grading on $\HFm(S^3)$; we describe this process below.

The differential $\d \co \CFKR(\cH) \to \CFKR(\cH)$ is defined to be
\[ \d \bfx = \sum_{\bfy \in \bbT_\alpha \cap \bbT_\beta} \sum_{\substack{\phi \in \pi_2(\bfx, \bfy) \\ \mu(\phi) = 1}} \# \cMhat(\phi) U^{n_w(\phi)} V^{n_z(\phi)} \; \bfy, \]
and is extended to all elements of $\CFKR(\cH)$ by $\F[U, V]$-linearity. Note that the differential preserves the Alexander grading. 

Setting $V=1$ and forgetting $\gr_V$ (that is, only considering the grading $\gr_U$), we recover $\CFm(S^3)$, whose homology is isomorphic to $\F[U]_{(0)}$, where the subscript $(0)$ now denotes $\gr_U(1)$; this determines the absolute $U$-grading. Note that setting $V=1$ corresponds to forgetting the $z$-basepoint. To determine the absolute $V$-grading, we simply reverse the roles of $U$ and $V$ in the above construction. That is, we set $U=1$, forget $\gr_U$, and only consider $\gr_V$, recovering $\CFm(S^3)$, whose homology is isomorphic to $\F[V]$ where $\gr_V(1) = 0$. This corresponds to forgetting the $w$-basepoint.

\begin{theorem}[{\cite[Theorem 3.1]{OSknots}}]
Let $\cH$ be a doubly pointed Heegaard diagram for a knot $K \subset S^3$. The chain homotopy type of $\CFKR(\cH)$ is an invariant of $K \subset S^3$.
\end{theorem}

\noindent Note that \cite[Theorem 3.1]{OSknots} is phrased in terms of filtered chain complexes; see \cite[Section 1.5]{Zemkeabsgr} for a description of the translation between filtered chain complexes and modules over $\F[U, V]$. We will often abuse notation and write $\CFKR(K)$ rather than $\CFKR(\cH)$.

\begin{example}\label{ex:LHT1}
Figure \ref{fig:HDtrefoil} shows a doubly pointed Heegaard diagram for the left-handed trefoil.
We have
\begin{align*}
	\d a &= Ub \\
	\d b &= 0 \\
	\d c &= Vb.	
\end{align*}
Setting $V=1$, we see that the homology, which is isomorphic to $\F[U]$, is generated by $[a+Uc]$, implying that $\gr_U(a) = \gr_U(Uc) = 0$. Setting $U=1$, we see that the homology, which is isomorphic to $\F[V]$, is generated by $[c+Va]$, implying that $\gr_V(c) = \gr_V(Va) = 0$. It follows that the generators $a, b, c$ have the following gradings:

\begin{center}
\begin{tabular}{*{10}{@{\hspace{10pt}}c}}
\hline
& && $\gr_U$ && $\gr_V$ && $A$  & \\
\hline
& $a$ && $0$ && $2$ && $-1$\\ 
& $b$ && $1$ && $1$ && $0$\\ 
& $c$ && $2$ && $0$ && $1$\\ 
\hline
\end{tabular}
\end{center}
Then $H_*(\CFKR(-T_{2,3})) \cong \F[U,V]_{(0,0)} \oplus \F_{(1,1)}$, where the subscript denotes $\gr=(\gr_U, \gr_V)$ of $1$. The $\F[U,V]$ summand is generated by $Va+Uc$ and the $\F$ summand by $b$.
\end{example}

\begin{exercise}\label{exer:flipUV}
Let $\cH=(\Sigma, \bfalpha, \bfbeta, w, z)$ be a doubly pointed Heegaard diagram for $K \subset S^3$. 
\begin{enumerate}
	\item\label{it:localize} Let 
	\[ (\CFKR(\cH) \otimes_{\F[V]} \F[V, V^{-1}], s)\] 
	denote the summand of $\CFKR(\cH) \otimes_{\F[V]} \F[V, V^{-1}]$ in Alexander grading $s$, thought of as an $\F[W]$-module, where $W=UV$. Note that multiplication by $UV$ preserves the Alexander grading on $\CFKR(\cH)$. Let $\CFm(\cH) = \CFm(\Sigma, \bfalpha, \bfbeta, w)$. 
	Show that there is an isomorphism of $\F[W]$-modules
	\[ (\CFKR(\cH) \otimes_{\F[V]} \F[V, V^{-1}], s) \cong \CFm(\cH), \] 
	where the right-hand side is viewed as a module over $\F[W]$, rather than $\F[U]$. (Hint: Consider the map $\CFm(\cH) \to (\CFKR(\cH) \otimes_{\F[V]} \F[V, V^{-1}], s)$ given by $W^n\bfx \mapsto U^nV^{n+s-A(\bfx)} \bfx$.)
	Conclude that 
	\[ H_*(\CFKR(\cH) \otimes_{\F[V]} \F[V, V^{-1}]) \cong \F[U, V, V^{-1}]. \] 

	\item Repeat part \eqref{it:localize} reversing the roles of $U$ and $V$.
	%Similarly, show that 
	%\[ H_*(\CFKR(\cH) \otimes_{\F[U]} \F[U, U^{-1}]) \cong \HFm(S^3) \otimes_\F \F[U, U^{-1}] \cong \F[U, U^{-1}, V], \]
	%where here $\HFm(S^3)$ is identified with the polynomial ring $\F[V]$ (rather than $\F[U]$).
\end{enumerate}
\end{exercise}

The knot Floer complex behaves nicely under connected sum, reversal, and mirroring. By \cite[Theorem 7.1]{OSknots}, we have that
\[ \CFKR(K_1 \# K_2) \simeq \CFKR(K_1) \otimes_{\F[U,V]} \CFKR(K_2). \] 
Let $K^r$ denote the reverse of $K$ and $mK$ the mirror. By \cite[Section 3]{OSknots}, we have that 
\begin{equation}\label{eq:mirror}
	\CFKR(mK) \simeq \CFKR(K)^*,
\end{equation}
where $C^* = \Hom_{\F[U,V]}(C, \F[U,V])$ and
\begin{equation}\label{eq:reverse}
	\CFKR(K^r) \simeq \CFKR(K).
\end{equation}

\begin{exercise}\label{exer:reverse}
Let $\cH = (\Sigma, \bfalpha, \bfbeta, w, z)$ be a doubly pointed Heegaard diagram for $K \subset S^3$. Show that $\cH_1 = (\Sigma, \bfalpha, \bfbeta, z, w)$ and $\cH_2 = (-\Sigma, \bfbeta, \bfalpha, w, z)$ are both diagrams for $K^r$. Show that $\CFKR(\cH)$ is isomorphic to $\CFKR(\cH_2)$ (cf. Exercise \ref{exer:conj}). Conclude that Equation \eqref{eq:reverse} holds.
\end{exercise}

\begin{remark}\label{rk:swapUV}
It follows from Equation \eqref{eq:reverse} and Exercise \ref{exer:reverse} that $\CFKR(K)$ is chain homotopy equivalent to complex $C'$ obtained from $\CFKR(K)$ by exchanging the roles of $U$ and $V$. (Note that one should then also exchange the roles of $\gr_U$ and $\gr_V$.)
\end{remark}

\begin{exercise}
Compute $\CFKR(T_{2,3})$ two ways: from a doubly pointed Heegaard diagram and by applying Equation \eqref{eq:mirror} to Example \ref{ex:LHT1}, and confirm that the two answers agree.
\end{exercise}

\begin{exercise}\label{exer:CFKRcomp}
Compute $\CFKR$ for the figure eight knot and for the torus knot $T_{3,4}$ using the doubly pointed Heegaard diagrams from Exercises \ref{exer:41} and \ref{exer:Tpq}.
\end{exercise}

\subsection{Algebraic variations}
As usual, let $\cH$ be a doubly pointed Heegaard diagram for a knot $K \subset S^3$. There are several algebraic modifications we may make to $\CFKR(\cH)$. Since the chain homotopy type of $\CFKR(\cH)$ is an invariant of the knot $K$, it follows that these algebraic modifications also yield knot invariants.

The first modification we consider is setting both $U = 0$ and $V = 0$, resulting in a bigraded chain complex $\CFK_\F(\cH)$ over the field $\F$. Setting $U=V=0$ is equivalent to requiring that $n_w(\phi) = n_z(\phi) = 0$ in the definition of the differential. We denote the homology of $\CFK_\F(\cH)$ by $\HFKhat(K)$, which is a bigraded vector space.

It is common to use $\gr_U$ and $A$ as the bigrading on $\HFKhat(K)$. (Of course, this is the same information as $\gr_U$ and $\gr_V$, since $A = \frac{1}{2}(\gr_U-\gr_V)$.) We write $\HFKhat_{m}(K, s)$ to denote the summand of $\HFKhat(K)$ with $\gr_U=m$ and $A=s$. The grading $\gr_U$ is often called the \emph{Maslov grading}.

\begin{example}
Setting $U =V = 0$ in Example \ref{ex:LHT1} results in 
\[ \d a = \d b = \d c = 0.\]
Thus, $\HFKhat(K)$ has dimension 3, generated by $a, b,$ and $c$, with gradings given in the table in Example \ref{ex:LHT1}. That is, 
\begin{align*}
	\HFKhat_{m}(K, s) = 		\begin{cases}
			\F \quad & \text{if } (m,s) = (0, -1), (1, 0), \text{ or } (2, 1) \\
			0 \quad & \text{otherwise}.
		\end{cases}
\end{align*}
\end{example}

\begin{theorem}[{\cite[Equation (1)]{OSknots}}]\label{thm:HFKcat}
The graded Euler characteristic of $\HFKhat(K)$ is equal to the Alexander polynomial of $K$:
\[ \Delta_K(t) = \sum_{m, s} (-1)^m \dim \HFKhat_{m}(K,s) \; t^s. \]
\end{theorem}

Recall from \cite{Kauffman} (see also \cite[Theorem 11.3]{OSintro}) that the Alexander polynomial of $K$ can be computed in terms of the Kauffman states of a diagram for $K$. Note that the Kauffman states of the left diagram in Figure \ref{fig:trefoilgenus4} are in bijection with the Heegaard Floer generators of the right diagram. This observation, together with a computation of the bigradings, is at the heart of the proof of Theorem \ref{thm:HFKcat}; see \cite[Sections 11-13]{OSintro} for details.

Another algebraic modification is to set a single variable, say $V$, equal to zero, resulting in a chain complex $\CFK_{\F[U]}(\cH)$ over the PID $\F[U]$. This corresponds to requiring that $n_z(\phi)=0$ in the definition of the differential. The homology of $\CFK_{\F[U]}$ is a $\F[U]$-module, denoted $\HFKm(K)$. As a finitely generated graded module over a PID, $\HFKm(K)$ is isomorphic to a direct sum of free summands and $U$-torsion summands as in Equation \eqref{eq:fingenmod}.

It is common to view $\HFKm(K)$ as bigraded by $\gr_U$ and $A$. The action of $U$ lowers $\gr_U$ by 2 and $A$ by $1$.

\begin{exercise}
Let $K \subset S^3$. Prove that $\HFKm(K) \otimes_\F \F[U, U^{-1}] \cong H_*(\CFK_{\F[U]}(K) \otimes_{\F[U]} \F[U, U^{-1}]) \cong \HFhat(S^3) \otimes_\F \F[U, U^{-1}] \cong \F[U, U^{-1}]$ and conclude that there is a unique free summand in $\HFKm(K)$. %(Hint: Show that $H_*(\CFK_{\F[U]}(K) \otimes_{\F[U]} \F[U, U^{-1}]) \cong \HFhat(S^3) \otimes_\F \F[U, U^{-1}]$.)
\end{exercise}

\begin{example}
Setting $V=0$ in Example \ref{ex:LHT1} results in the free $\F[U]$-module generated by $a, b,$ and $c$ with differential
\begin{align*}
	\d a &= Ub \\
	\d b &= 0 \\
	\d c &= 0.	
\end{align*}
Hence
\[ \HFKm(K) \cong \F[U]_{(2)} \oplus \F_{(1)} \]
where $[c]$ is a generator for the $\F[U]$-summand, $[b]$ is a generator for the $\F$-summand, and the subscript denotes $\gr_U$.
\end{example}

There are other algebraic modifications one may consider, such as setting $U^n=0$ or $UV=0$.%; the latter modification will be of use in Section \ref{sec:surgery}.

\subsection{Computations}
How does one compute $\CFKR$ in practice? For small crossing knots, $\CFKR$ can be computed via grid diagrams \cite{MOS, MOST}; see \cite{BaldwinGillam} for a table of $\HFKhat$ for knots up to 12 crossings computed using grid diagrams. For an excellent textbook on the subject of grid diagrams, see \cite{OSSgrid}. The invariant $\CFK_{\F[U,V]/(UV=0)}$ can be algorithmically computed following \cite{OSborderedalgebras} (see also \cite{OSmatchings}); such computations are significantly faster than computations with grid diagrams. At the time of writing, a computer implementation of this algorithm is available at \url{https://web.math.princeton.edu/~szabo/HFKcalc.html}.

For certain special families of knots, we can compute $\CFKR$  directly from the definition (in the case of $(1,1)$-knots, described below) or from other easier to compute knot invariants such as the Alexander polynomial and signature (in the case of alternating knots and knots admitting L-space surgeries).

A knot in $S^3$ that admits a genus 1 doubly pointed Heegaard diagram is called a \emph{$(1, 1)$-knot}. For a $(1,1)$-knot $K$, the complex $\CFKR(K)$ can be computed by counting embedded disks in the universal cover of $\Sigma=T^2$, similar to Example \ref{ex:LHT1}; see \cite{GodaMM}. 

If $K$ is alternating (or more generally, quasi-alternating; see \cite[Definition 3.1]{OSdbc}), then \cite[Theorem 1.3]{OSalternating} states that $\HFKhat(K)$ is completely determined by the Alexander polynomial and signature of $K$. Moreover, \cite[Lemma 7]{Petkovacables} (which is completely algebraic) states that if $K$ is alternating, then $\HFKhat(K)$ completely determines the chain homotopy type of $\CFKR(K)$.

\begin{exercise}
Let $K$ be an alternating knot. A key ingredient in the proof of \cite[Theorem 1.3]{OSalternating} is that if $\HFKhat_m(K,s) \neq 0$, then $m=s+\frac{\sigma(K)}{2}$, where $\sigma(K)$ denotes the signature of $K$. (If $\HFKhat(K)$ is supported on a single diagonal with respect to the Maslov and Alexander gradings, we say $K$ is \emph{homologically thin}.) Show that this fact combined with Theorem \ref{thm:HFKcat} completely determines the bigraded vector space $\HFKhat(K)$ when $K$ is an alternating knot.
\end{exercise}

If $K$ admits a lens space surgery (or more generally, an L-space surgery), it follows from \cite[Theorem 1.2]{OSlens}  that $\CFKR(K)$ is completely determined by the Alexander polynomial of $K$, as follows. If $K$ admits an L-space surgery, then the non-zero coefficients in $\Delta_K(t)$ are all $\pm 1$ and they alternate in sign \cite[Corollary 1.3]{OSlens}. Let
\[ \Delta_K(t) = \sum_{i=0}^n (-1)^i t^{a_i}, \]
for some decreasing sequence $(a_i)$ and even $n$. Let $b_i = a_{i} - a_{i-1}$. If $K$ admits a positive L-space surgery, then $\CFKR(K)$ is generated by $x_0, \dots, x_n$ where for $i$ odd,
\[ \d x_i = U^{b_i} x_{i-1} + V^{b_{i+1}} x_{i+1}, \]
and for $i$ even, $\d x_i = 0$. The absolute grading is determined by $\gr_U(x_0) = 0$ and $\gr_V(x_n) = 0$. (Note that there is no loss of generality in considering only positive L-space surgeries. Indeed, if $K$ admits a negative L-space surgery, then $mK$ admits a positive L-space surgery and one can apply Equation \eqref{eq:mirror}.)

\begin{exercise}
Compute $\CFKR(T_{3,4})$ using the fact that $T_{3,4}$ admits a positive L-space surgery and the above description of $\CFKR(K)$ in terms $\Delta_K(t)$ for knots admitting L-space surgeries. Compare with Exercise \ref{exer:CFKRcomp}.
\end{exercise}

\begin{exercise}
Suppose $K$ admits a positive L-space surgery. Express $\HFKhat(K)$ in terms of $\Delta_K(t)$, and verify that $\HFKhat(K)$ satisfies Theorem \ref{thm:HFKcat}.
\end{exercise}

\noindent (For the relationship between $(1, 1)$-knots and L-space knots, see \cite{GreeneLV}.) 

%%%%%%%%%%%%%%%%%%%%%%%%%%%%%%
%%%%%%%%%%%%%%%%%%%%%%%%%%%%%%

\section{Heegaard Floer homology of knot surgery}\label{sec:surgery}

\subsection{Large surgery}
In this section, we discuss the relationship between the knot Floer complex $\CFKR(K)$ and $\HFm(S^3_n(K))$, where $S^3_n(K)$ denotes $n$-surgery on $K \subset S^3$.

We begin with some observations about $\CFKR(K)$, which is a chain complex over $\F[U, V]$. Let $W=UV$. Note that multiplication by $W$ preserves the Alexander grading. Hence as a chain complex over $\F[W]$, the complex $\CFKR(K)$ splits as a direct sum over the Alexander grading. (However, note that neither multiplication by $U$ nor by $V$ respects this splitting.)

Following \cite[Section 4]{OSknots}, one may identify $\spinc$-structures on $S^3_n(K)$ with $\Z/n\Z$. Recall that $\HFm(S^3_n(K))$ is a module over a polynomial ring in a single variable.%; we denote this variable by $W$ (rather than $U$, as in Section \ref{sec:HF}).

\begin{theorem}[{\cite[Theorem 4.4]{OSknots}, cf. \cite[Section 4]{RasmussenThesis}}]\label{thm:HFsurgery}
Let $n \geq 2g(K)-1$ and $|s| \leq \lfloor \frac{n}{2} \rfloor$. Then 
\[ \HFm(S^3_n(K), [s]) \cong H_*(\CFKR(K, s)) \]
as relatively $\Z$-graded modules over a polynomial ring in a single variable $W$. That is, on the left-hand side, $W =U$ while on the right-hand side, $W = UV$. The relative grading on the right-hand side may be taken to be either $\gr_U$ or $\gr_V$.
\end{theorem}

\begin{remark}
See \cite[Corollary 4.2]{OSknots} for the absolutely graded version of Theorem \ref{thm:HFsurgery}.
\end{remark}

\begin{example}\label{ex:+1LHT}
Let $Y=S^3_{+1}(-T_{2,3})$. (It follows from \cite[Proposition 3.1]{Moser} that $S^3_{+1}(-T_{2,3}) \cong -\Sigma(2,3,7)$.) We will use Theorem \ref{thm:HFsurgery} to compute $\HFm(Y)$. Since $Y$ is an integer homology sphere, there is a unique $\spinc$-structure on $Y$. From Example \ref{ex:LHT1}, we have that $\CFKR(-T_{2,3}, 0)$ is generated over $\F[W]$, where $W=UV$, by
\[ Va, b, Uc. \] 
The differential is given by
\begin{align*}
	\d (Va) &= W \cdot b \\
	\d b &= 0 \\
	\d (Uc) &= W \cdot b.
\end{align*}
Note that the elements $Va$ and $Uc$ are in the same relative grading, while the relative grading of $b$ is one greater than the relative grading of $Va$.
We have that 
\[ \HFm(Y) \cong H_*(\CFKR(-T_{2,3}, 0)) \cong \F[W]_{(0)} \oplus \F_{(1)}, \]
where the $ \F[W]$-summand is generated by $Va+Uc$ and the $\F$-summand is generated by $b$. (The absolute gradings are computed following \cite[Corollary 4.2]{OSknots}.)
\end{example}

\begin{example}
Let $Y=S^3_{+3}(-T_{2,3})$. By Theorem \ref{thm:HFsurgery}, 
\[ \HFm(Y, [s]) \cong H_*(\CFKR(-T_{2,3}, s) \]
for $s=-1, 0, 1$.
By Example \ref{ex:+1LHT}, we have 
\[ \HFm(Y, [0]) \cong  H_*(\CFKR(-T_{2,3}, 0)) \cong \F[W] \oplus \F. \] 
From Example \ref{ex:LHT1}, we have that $\CFKR(-T_{2,3}, -1)$ is generated over $\F[W]$ by
\[ a, Ub, U^2c. \] 
The differential is given by
\begin{align*}
	\d a &= Ub \\
	\d (Ub) &= 0 \\
	\d (U^2c) &= W \cdot Ub.
\end{align*}
Hence $H_*(\CFKR(-T_{2,3}, -1)) \cong \F[W]$, generated by $U^2 c + Wa$. Similarly, $H_*(\CFKR(-T_{2,3}, 1)) \cong \F[W]$, generated by $V^2 a + W c$; we leave this calculation to the reader. See Figure \ref{fig:-T23As}.
\end{example}

\begin{figure}%[H]
\subfigure[]{
\begin{tikzpicture}[scale=1.25]
	\node at (0, 0) (a) {\lab{a}};
	\node at (-1, 0) (Ub) {\lab{Ub}};
	\node at (-2, 0) (U2c) {\lab{U^2c}};
	\node at (-1, -1) (UVa) {\lab{UVa}};
	\node at (-2, -1) (U2Vb) {\lab{U^2Vb}};
	\node at (-3, -1) (U3Vc) {\lab{U^3Vc}};
	\node at (-2, -2) (U2V2a) {\lab{U^2V^2a}};
	\node at (-3, -2) (U3V2b) {\lab{U^3V^2b}};

	\draw[<-] (Ub) to node[]{} (a);
	\draw[<-] (U2Vb) to node[]{} (UVa);
	\draw[<-] (U3V2b) to node[]{} (U2V2a);
	\draw[<-] (U2Vb) to node[]{} (U2c);
	\draw[<-] (U3V2b) to node[]{} (U3Vc);
	
	\node at (-3.6, -2.3) (dots) {$\iddots$};
\end{tikzpicture}\label{subfig:a}
}
\hspace{15pt}
\subfigure[]{
\begin{tikzpicture}[scale=1.25]
	\node at (0, 0) (b) {\lab{b}};
	\node at (-1, -1) (UVb) {\lab{UVb}};
	\node at (-1, 0) (Uc) {\lab{Uc}};
	\node at (0, -1) (Va) {\lab{Va}};
	\node at (-2, -2) (U2V2b) {\lab{U^2V^2b}};
	\node at (-2, -1) (U2Vc) {\lab{U^2Vc}};
	\node at (-1, -2) (UV2a) {\lab{UV^2a}};
	
	\draw[<-] (UVb) to node[]{} (Va);
	\draw[<-] (UVb) to node[]{} (Uc);
	\draw[<-] (U2V2b) to node[]{} (UV2a);
	\draw[<-] (U2V2b) to node[]{} (U2Vc);
	
	\node at (-2.6, -2.3) (dots) {$\iddots$};
\end{tikzpicture}
}
\subfigure[]{
\begin{tikzpicture}[scale=1.25]
	\node at (0, 1) (c) {\lab{c}};
	\node at (0, 0) (Vb) {\lab{Vb}};
	\node at (-1, -1) (U2Vb) {\lab{U^2Vb}};
	\node at (-1, 0) (UVc) {\lab{UVc}};
	\node at (0, -1) (V2a) {\lab{V^2a}};
	\node at (-2, -1) (U2V2c) {\lab{U^2V^2c}};

	\draw[<-] (Vb) to node[]{} (c);
	\draw[<-] (U2Vb) to node[]{} (UVc);
	\draw[<-] (U2Vb) to node[]{} (V2a);
	
	\node at (-2.6, -1.3) (dots) {$\iddots$};
\end{tikzpicture}
}
  \vspace{8pt}\caption{Different Alexander graded summands of $\CFKR(-T_{2,3})$.  Top left, $A_{-1} = \CFKR(-T_{2,3}, -1)$. Top right, $A_{0} = \CFKR(-T_{2,3}, 0)$. Bottom, $A_{1} = \CFKR(-T_{2,3}, 1)$.}
    \label{fig:-T23As}
  \end{figure}
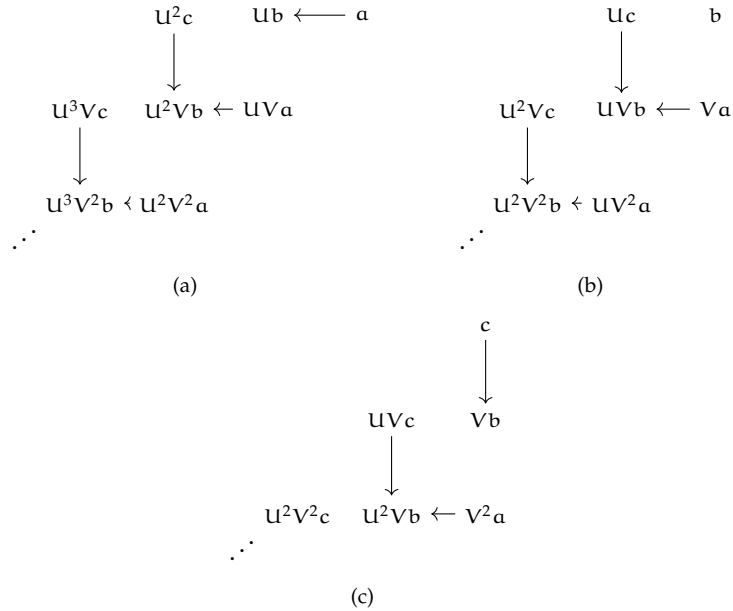

The proof of Theorem \ref{thm:HFsurgery} relies on relating the Heegaard diagrams for $(S^3, K)$ and $S^3_n(K)$; see, for example, Figures \ref{fig:trefoilgenus4} and \ref{fig:trefoilsurgery}. The rough idea is that for each generator of the Heegaard diagram for $(S^3, K)$ and for each $\spinc$-structure on $S^3_n(K)$, there is a canonical ``nearest'' generator obtained by replacing the intersection point on the meridian with a nearby intersection point on the $n$-framed longitude. The remainder of the proof relies on the relationship between $\spinc$-structures and the Alexander grading, as well as a count of holomorphic triangles in a Heegaard triple. See \cite[Section 4]{OSknots} for more details.

\begin{exercise}
Compute $\HFm(S^3_{+1}(4_1))$, where $4_1$ denotes the figure eight knot, using $\CFKR(4_1)$ as computed in Exercise \ref{exer:CFKRcomp}.
\end{exercise}

\begin{exercise}\label{exer:+5T34}
Compute $\HFm(S^3_{+5}(T_{3,4}))$ using $\CFKR(T_{3,4})$ as computed in Exercise \ref{exer:CFKRcomp}.
\end{exercise}

\subsection{Integer surgery}
Note that Theorem \ref{thm:HFsurgery} requires that  surgery coefficient $n$ to be greater than or equal to $2g(K)-1$. In \cite[Theorem 1.1]{OSinteger}, Ozsv\'ath-Szab\'o provide a recipe for computing the Heegaard Floer homology of any integer surgery along $K \subset S^3$; they improve this to a formula for rational surgery in \cite{OSrational}.

In these notes, we will work with the minus flavor, as in \cite[Theorem 1.1]{MOlink}. One disadvantage to working with the minus flavor is that one must work with completed coefficients (cf. Remark \ref{rk:completions}), that is, we work over the power series rings $\F[[U]]$ and $\F[[U, V]]$. To this end, for a pointed Heegaard diagram $\cH$ for $Y$, let
\[ \CFm_{\F[[U]]}(\cH) = \CFm(\cH) \otimes_{\F[U]} \F[[U]] \]
and let $\HFm_{\F[[U]]}(Y) = H_*(\CFm_{\F[[U]]}(\cH))$. Note that since $\F[[U]]$ is flat over $\F[U]$, we have that $\HFm_{\F[[U]]}(Y)  \cong \HFm(Y) \otimes_{\F[U]} \F[[U]]$.
Similarly, let
\[ \CFK_{\F[[U,V]]}(K) = \CFKR(K) \otimes_{\F[U,V]} \F[[U,V]].  \]
Let
\[ A_s = \CFK_{\F[[U,V]]}(K, s) \]
denote the part of $\CFK_{\F[[U,V]]}(K)$ in Alexander grading $s$. (Note that this is not quite a grading in the usual sense, as $\CFK_{\F[[U,V]]}(K)$ is a direct product rather than a direct sum of its homogenously graded pieces. We will abuse notation and still refer to $s$ as the Alexander grading, and similarly for $\gr_U$ and $\gr_V$.) We have that
\[ \CFK_{\F[[U,V]]}(K) \cong \prod_s A_s. \]

Let 
\[ B_s = (\CFK_{\F[U, V]}(K) \otimes_{\F[U,V]} \F[[U,V, V^{-1}]], s) \]
denote the part of $\CFK_{\F[U, V]}(K) \otimes_{\F[U,V]} \F[[U,V, V^{-1}]]$ in Alexander grading $s$. By Exercise \ref{exer:flipUV}, we have that $B_s \simeq \CFm_{\F[[W]]}(S^3)$ for every $s$, where $W=UV$. Note that 
\[ V \co B_s \to B_{s+1} \qquad \textup{ and } \qquad V^{-1} \co B_{s+1} \to B_s. \]
Since the composition of these two maps is the identity, we see that $V$ maps $B_s$ isomorphically to $B_{s+1}$. We have
\[ \CFK_{\F[U, V]}(K) \otimes_{\F[U,V]} \F[[U,V, V^{-1}]] \cong \prod_s B_s. \]
with $V, V^{-1},$ and $W$ acting as follows
\begin{equation*}
\label{pic:exact}
\begin{tikzpicture}[baseline=(current  bounding  box.center)]
\node(-2)at(-1.2,-0.2){};
\node(-1)at(-1.2, .2){};
\node(0)at(-1.4,0){${} \dots$};
\node(1)at(0,0){$B_{s-1}$};
\node(2)at (1.5, 0){$B_s$};
\node(3)at (3, 0){$B_{s+1}$};
\node(4)at(4.4,0){$\dots$};
\node(5)at(4.2,-0.2){};
\node(6)at(4.2, .2){};
\path[->, bend left=35](1)edge node[above]{\lab{V}}(2);
\path[<-, bend right=35](1)edge node[below]{\lab{V^{-1}}}(2);
\path[->, bend left=35](2)edge node[above]{\lab{V}}(3);
\path[<-, bend right=35](2)edge node[below]{\lab{V^{-1}}}(3);
\path[<-, bend right=35](3)edge node[below]{\lab{V^{-1}}}(5);
\path[->, bend left=35](3)edge node[above]{\lab{V}}(6);
\path[->, bend left=35](1)edge node[below]{\lab{V^{-1}}}(-2);
\path[<-, bend right=35](1)edge node[above]{\lab{V}}(-1);

\drawloop[->,stretch=1]{1}{50}{130} node[pos=0.5,above]{\lab{W}};
\drawloop[->,stretch=1]{2}{50}{130} node[pos=0.5,above]{\lab{W}};
\drawloop[->,stretch=1]{3}{50}{130} node[pos=0.5,above]{\lab{W}};
\end{tikzpicture}
\end{equation*}

By Remark \ref{rk:swapUV}, we have that $\CFK_{\F[U, V]}(K) \otimes_{\F[U,V]} \F[[U,V, V^{-1}]]$ is chain homotopy equivalent to $\CFK_{\F[U, V]}(K) \otimes_{\F[U,V]} \F[[U, U^{-1},V]]$ after exchanging the roles of $U$ and $V$. (Note that this chain homotopy equivalence reverses the Alexander grading.) Moreover, in any fixed Alexander grading, both complexes are homotopy equivalent to $B_s$, so let 
 \[ \phi_s  \co (\CFK_{\F[U,V]}(K) \otimes_{\F[U,V]} \F[[U, U^{-1},V]], s) \to (\CFK_{\F[U,V]}(K) \otimes_{\F[U,V]} \F[[U,V, V^{-1}]], s) \]
denote an Alexander grading-preserving chain homotopy equivalence between these two $\F[UV]$-modules.
Let $\phi = \prod_s \phi_s$. (Note that $\phi$ is not $\F[U,V]$-equivariant, although it is $\F[UV]$-equivariant.)

Let 
\[ \iota_V \co \CFK_{\F[[U,V]]}(K) \to \CFK_{\F[U, V]}(K) \otimes_{\F[U,V]} \F[[U,V, V^{-1}]] \]
and
\[ \iota_U \co \CFK_{\F[[U,V]]}(K) \to \CFK_{\F[U, V]}(K) \otimes_{\F[U,V]} \F[[U, U^{-1}, V]] \]
denote inclusion.
Moreover, since multiplication by $V$ is invertible in $\CFK_{\F[U,V]}(K) \otimes_{\F[U,V]} \F[[U,V, V^{-1}]]$, we have that 
\[ V^n \co \CFK_{\F[U, V]}(K) \otimes_{\F[U,V]} \F[[U,V, V^{-1}]] \to \CFK_{\F[U, V]}(K) \otimes_{\F[U,V]} \F[[U,V, V^{-1}]] \] 
is a (relatively graded) isomorphism. Note that $V^n|_{B_s} \co B_s \to B_{s+n}$.

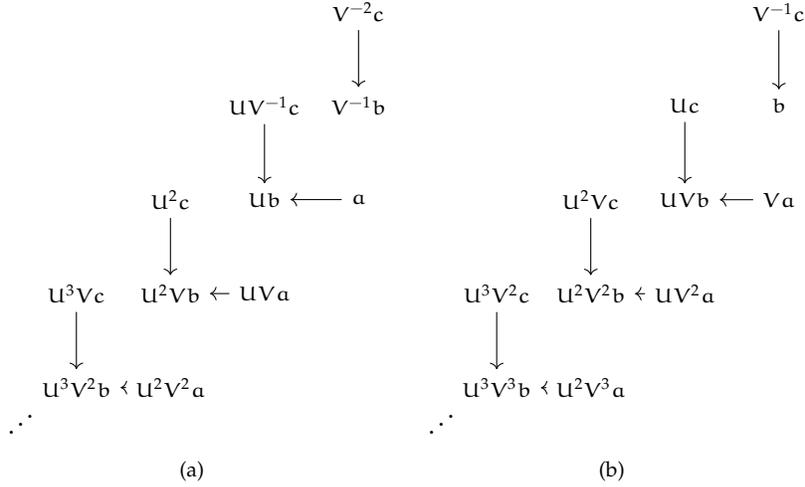
\begin{figure}%[H]
\subfigure[]{
\begin{tikzpicture}[scale=1.25]
	\node at (0, 2) (V-2c) {\lab{V^{-2}c}};	
	\node at (0, 1) (V-1b) {\lab{V^{-1}b}};	
	\node at (-1, 1) (UV-1c) {\lab{UV^{-1}c}};	
	\node at (0, 0) (a) {\lab{a}};
	\node at (-1, 0) (Ub) {\lab{Ub}};
	\node at (-2, 0) (U2c) {\lab{U^2c}};
	\node at (-1, -1) (UVa) {\lab{UVa}};
	\node at (-2, -1) (U2Vb) {\lab{U^2Vb}};
	\node at (-3, -1) (U3Vc) {\lab{U^3Vc}};
	\node at (-2, -2) (U2V2a) {\lab{U^2V^2a}};
	\node at (-3, -2) (U3V2b) {\lab{U^3V^2b}};

	\draw[->] (V-2c) to node[]{} (V-1b);
	\draw[<-] (Ub) to node[]{} (UV-1c);
	\draw[<-] (Ub) to node[]{} (a);
	\draw[<-] (U2Vb) to node[]{} (UVa);
	\draw[<-] (U3V2b) to node[]{} (U2V2a);
	\draw[<-] (U2Vb) to node[]{} (U2c);
	\draw[<-] (U3V2b) to node[]{} (U3Vc);
	
	\node at (-3.6, -2.3) (dots) {$\iddots$};
\end{tikzpicture}
}
\subfigure[]{
\begin{tikzpicture}[scale=1.25]
	\node at (0, 1) (V-1c) {\lab{V^{-1}c}};
	\node at (0, 0) (b) {\lab{b}};
	\node at (-1, -1) (UVb) {\lab{UVb}};
	\node at (-1, 0) (Uc) {\lab{Uc}};
	\node at (0, -1) (Va) {\lab{Va}};
	\node at (-2, -2) (U2V2b) {\lab{U^2V^2b}};
	\node at (-2, -1) (U2Vc) {\lab{U^2Vc}};
	\node at (-1, -2) (UV2a) {\lab{UV^2a}};
	\node at (-3, -3) (U3V3b) {\lab{U^3V^3b}};
	\node at (-3, -2) (U3V2c) {\lab{U^3V^2c}};
	\node at (-2, -3) (U2V3a) {\lab{U^2V^3a}};
	
	\draw[<-] (b) to node[]{} (V-1c);
	\draw[<-] (UVb) to node[]{} (Va);
	\draw[<-] (UVb) to node[]{} (Uc);
	\draw[<-] (U2V2b) to node[]{} (UV2a);
	\draw[<-] (U2V2b) to node[]{} (U2Vc);
	\draw[<-] (U3V3b) to node[]{} (U2V3a);
	\draw[<-] (U3V3b) to node[]{} (U3V2c);
	
	\node at (-3.6, -3.3) (dots) {$\iddots$};
\end{tikzpicture}
}
%\subfigure[]{
%\begin{tikzpicture}[scale=1.25]
%	\node at (0, 1) (c) {\lab{c}};
%	\node at (0, 0) (Vb) {\lab{Vb}};
%	\node at (-1, -1) (U2Vb) {\lab{U^2Vb}};
%	\node at (-1, 0) (UVc) {\lab{UVc}};
%	\node at (0, -1) (V2a) {\lab{V^2a}};
%	\node at (-2, -1) (U2V2c) {\lab{U^2V^2c}};

%	\draw[<-] (Vb) to node[]{} (c);
%	\draw[<-] (U2Vb) to node[]{} (UVc);
%	\draw[<-] (U2Vb) to node[]{} (V2a);
	
%	\node at (-2.6, -1.3) (dots) {$\iddots$};
%\end{tikzpicture}
%}
  \vspace{8pt}\caption{Different Alexander graded pieces of $\CFK_{\F[U,V]}(-T_{2,3}) \otimes_{\F[U,V]} \F[[U,V, V^{-1}]] \cong \prod_s B_s$. Left, $B_{-1}$. Right, $B_0$. Note that the two complexes are isomorphic, with the isomorphism provided by multiplication by $V^{\pm 1}$.}
    \label{fig:}
  \end{figure}

Consider the chain map
\[ D_n \co \CFK_{\F[[U,V]]}(K) \to \CFK_{\F[U, V]}(K) \otimes_{\F[U,V]} \F[[U,V, V^{-1}]] \]
where $D_n = \iota_V + V^n \circ \phi \circ \iota_U$.

Recall that given two chain complexes $(X, \d_X), (Y, \d_Y)$ and a chain map $f \co X \to Y$, the \emph{mapping cone} of $f$ is the chain complex $\Cone(f) = X \oplus Y$ with the differential $\d(x, y) = (\d_X x, f(x) + \d_Y y)$. (Note that we are working over a ring of characteristic two; otherwise, one needs to insert some appropriate minus signs in the definition of the mapping cone.)

\begin{exercise}
Let $f \co X \to Y$ be a chain map. Show that $(\Cone(f), \d)$ is a chain complex.
\end{exercise}

The surgery formula expresses $\HFm(S^3_n(K))$ in terms of the mapping cone of $D_n$.

\begin{theorem}[{\cite[Theorem 1.1]{MOlink}, cf. \cite[Theorem 1.1]{OSinteger}}] \label{thm:mappingcone}
We have the following isomorphism of $\F[W]$-modules
\[ \HFm(S^3_n(K)) \otimes_{\F[W]} \F[[W]] \cong H_*(\Cone (D_n)),\]
where on the left-hand side $\HFm(S^3_n(K)$ is viewed as a module over $W=U$ and on the right-hand side, $W=UV$.
\end{theorem}

There is a similar formula for rational surgeries; see \cite{OSrational}. Note that 
\[ \iota_V|_{A_s} \co A_s \to B_s \qquad \textup{ and } \qquad V^n \circ \phi \circ \iota_U|_{A_s}  \co A_s \to B_{s+n}. \]
Hence 
\[ D_n|_{A_s} \co A_s \to B_s \oplus B_{s+n}, \]
and so $D_n$ may be depicted as follows:
\[
\begin{tikzpicture}[scale=1]
	\draw (-2.5,0) node [] (A-1) {$A_{s-n}$};
	\draw (0,0) node [] (A0) {$A_s$};
	\draw (2.5,0) node [] (A1) {$A_{s+n}$};

	%\draw (-2.5,-1.9) node [] (B-1) {${}$};	
	\draw (0,-2) node [] (B0) {$B_s$};
	\draw (2.5,-2) node [] (B1) {$B_{s+n}$};
	%\draw (4.8,-1.8) node [] (B2) {${}$};

	\draw (-2.5,-1) node [] {$\dots$};	
	\draw (3.75,-1) node [] {$\dots$};
	%\draw (-3.25,-0.85) node [] {$\dots$};	
	%\draw (5.25,-0.85) node [] {$\dots$};

	%\draw [->] (A-1) -- (B-1) node [midway, right] {\tiny${\iota_V}$};	
	\draw [->] (A-1) -- (B0) node [midway, right] {\tiny${V^n \phi \iota_U}$};
	\draw [->] (A0) -- (B0) node [midway, right] {\tiny${\iota_V}$};
	\draw [->] (A0) -- (B1) node [midway, right] {\tiny${V^n \phi \iota_U}$};
	\draw [->] (A1) -- (B1) node [midway, right] {\tiny${\iota_V}$};
	%\draw [->] (A1) -- (B2) node [midway, right] {\tiny${V^n \phi \iota_U}$};

\end{tikzpicture}
\]

\begin{exercise}\label{exer:truncation}
Use \eqref{eq:genus} to conclude that for $s \geq g(K)$, the maps
\[ \iota_V \co A_s \to B_s \qquad \textup{ and } \qquad V^n \circ \phi \circ \iota_U \co A_{-s} \to B_{-s+n} \]
are homotopy equivalences. Use this to conclude that $\Cone(D_1)$ is homotopy equivalent to
\[
\begin{tikzpicture}[scale=1]
	\draw (-2.5,0) node [] (A1-g) {$A_{1-g}$};
	\draw (0,0) node [] (A2-g) {$A_{2-g}$};
	\draw (0,-2) node [] (B2-g) {$B_{2-g}$};
	\draw (2.3,-1.8) node [] (B3-g) {$\; $};	
	
	\draw [->] (A1-g) -- (B2-g) node [midway, right] {\tiny${V \phi \iota_U}$};
	\draw [->] (A2-g) -- (B2-g) node [midway, right] {\tiny${\iota_V}$};
	\draw [->] (A2-g) -- (B3-g) node [midway, right] {};

	\draw (3.3,-1) node [] {$\dots$};

	\draw (5,0) node [] (Ag-2) {$A_{g-2}$};
	\draw (7.5,0) node [] (Ag-1) {$A_{g-1}$};
	\draw (7.5,-2) node [] (Bg-1) {$B_{g-1}$};
	\draw (5,-1.8) node [] (Bg-2) {$\; $};

	\draw [->] (Ag-2) -- (Bg-1) node [midway, right] {\tiny${V \phi \iota_U}$};
	\draw [->] (Ag-1) -- (Bg-1) node [midway, right] {\tiny${\iota_V}$};
	\draw [->] (Ag-2) -- (Bg-2) node [midway, right] {};

\end{tikzpicture}
\]
and that $\Cone(D_{-1})$ is homotopy equivalent to
\[
\begin{tikzpicture}[scale=1]
	\draw (-2.5,0) node [] (A1-g) {$A_{1-g}$};
	\draw (-0.2,-0.2) node [] (A2-g) {$\; $};
	\draw (-2.5,-2) node [] (B1-g) {$B_{1-g}$};
	\draw (-5,-2) node [] (B-g) {$B_{-g}$};	
	
	\draw [->] (A1-g) -- (B-g) node [midway, left] {\tiny${V^{-1} \phi \iota_U}$};
	\draw [->] (A1-g) -- (B1-g) node [midway, right] {\tiny${\iota_V}$};
	\draw [->] (A2-g) -- (B1-g) node [midway, right] {};

	\draw (0.8,-1) node [] {$\dots$};

	\draw (2.5,-0.2) node [] (Ag-2) {$\;$};
	\draw (5,0) node [] (Ag-1) {$A_{g-1}$};
	\draw (5,-2) node [] (Bg-1) {$B_{g-1}.$};
	\draw (2.5,-2) node [] (Bg-2) {$B_{g-2}$};

	\draw [->] (Ag-1) -- (Bg-2) node [midway, left] {\tiny${V^{-1} \phi \iota_U}$};
	\draw [->] (Ag-1) -- (Bg-1) node [midway, right] {\tiny${\iota_V}$};
	\draw [->] (Ag-2) -- (Bg-2) node [midway, right] {};

\end{tikzpicture}
\]
Similar statements hold for arbitrary positive (respectively negative) surgery coefficients.
\end{exercise}

\begin{remark}
We consider the knot Floer complex as a module over the bigraded ring $\F[U, V]$, whereas Ozsv\'ath-Szab\'o originally defined the knot Floer complex as a $\Z$-filtered chain complex over $\F[U]$. The two formulations are equivalent, as described in \cite[Section 1.5]{Zemkeabsgr}. Indeed, our notation $A_s$ and $B_s$ in this section lines up with the (minus flavor over the power series ring) of the complexes $A_s$ and $B_s$ in \cite{OSinteger}, while our $\iota_V$ (respectively $V^n \phi \iota_U$) is denoted $v$ (respectively $h$) in \cite{OSinteger}.
\end{remark}

\begin{remark}
Theorem \ref{thm:mappingcone} has analogous versions for the plus and hat flavors of Heegaard Floer homology. Both the plus and hat flavors have the advantage that one may pass to homology before taking the mapping cone. For the hat flavor, this is because the chain complexes are all vector spaces, while for the plus flavor, this is because the maps $\iota_{V, *}$ and $(V^n \circ \phi \circ \iota_U)_*$ are surjective for knots in $S^3$. See \cite[Section 5]{OSinteger} for some sample calculations. 
\end{remark}

\begin{exercise}
Compute $\HFm(S^3_{+1}(T_{3,4}))$ using Exercise \ref{exer:truncation} and the computation from Exercise \ref{exer:+5T34}.
\end{exercise}

\subsection{Applications}
We conclude with some applications of the mapping cone formula.

Recall the \emph{Cosmetic Surgery Conjecture}, which posits that for a nontrivial knot $K \subset S^3$ and $r, r' \in \Q$, if $S^3_r(K)$ and $S^3_{r'}(K)$ are homeomorphic as oriented manifolds, then $r=r'$. Using the mapping cone formula, Ni-Wu \cite{NiWu} prove the following:

\begin{theorem}[{\cite[Theorem 1.2]{NiWu}}]
Suppose $K$ is a nontrivial knot in $S^3$ such that $S^3_r(K) \cong S^3_{r'}(K)$ as oriented manifolds where $r, r'$ are distinct rational numbers. Then $r = -r'$ and $r$ is of the form $p/q$ where $p,q$ are relatively prime integers with $q^2 \equiv -1 \pmod p$.
\end{theorem}

\begin{remark}
Hanselman \cite[Theorem 2]{Hanselmancosmetic} recently improved this result to $r= \pm 2$ or $\pm 1/q$, using bordered Floer homology, a version of Heegaard Floer homology for 3-manifolds with parametrized boundary. Hanselman's result relies on a reinterpretation of bordered Floer homology as immersed curves \cite{HanselmanRasmussenWatson}.
\end{remark}

Heegaard Floer homology can also be used to obstruct manifolds from being obtained by surgery on a knot in $S^3$. Lickorish \cite{Lickorish} and Wallace \cite{Wallace} proved that every closed oriented 3-manifolds can be obtained by surgery on a link in $S^3$. However, since $H_1(S^3_{p/q}(K)) \cong \Z/p\Z$, it follows that if $H_1(Y; \Z)$ is not cyclic,  then $Y$ cannot be obtained by surgery on a knot in $S^3$. For example, $H_1(\R P^3 \# \R P^3; \Z) \cong \Z/2\Z \oplus \Z/2\Z$, hence $\R P^3 \# \R P^3$ is not surgery on any knot in $S^3$.

This $H_1$ obstruction vanishes for integer homology spheres. Auckly \cite{Auckly} gave an example of a hyperbolic integer homology sphere that cannot be obtained via surgery on a knot in $S^3$; moreover, his techniques can be used to give infinitely many such examples. Using the mapping cone formula, Hom-Karakurt-Lidman \cite{HomKarakurtLidman} prove the following:

\begin{theorem}[{\cite[Theorem 1.1]{HomKarakurtLidman}}]\label{thm:HKL}
The Brieskorn homology spheres $\Sigma(n, 2n-1, 2n+1)$, $n \geq 8$, $n$ even, cannot be obtained by surgery on a knot in $S^3$.
\end{theorem}

The proof of Theorem \ref{thm:HKL} relies on a surgery obstruction which relates the $d$-invariant of surgery on a knot in $S^3$ with $\HFred$; see \cite[Theorem 1.2]{HomKarakurtLidman} for the precise statement.

The mapping cone formula has also played a role in results that make no reference to Dehn surgery. For example, one could ask whether there exist simply-connected, positive definite symplectic 4-manifolds. Recall that a manifold is \emph{geometrically simply-connected} if it admits a handle decomposition with no 1-handles.

\begin{theorem}[{\cite[Theorem 1.1]{HomLidmanposdef}}]\label{thm:HLposdef}
Let $X$ be a closed, geometrically simply-connected 4-manifold with $b^+_2(X) \geq 2$. If the intersection form of $X$ is positive definite, then $X$ is not symplectic.
\end{theorem}

\begin{remark}
Yasui \cite{Yasui} has recently obtained another proof of Theorem \ref{thm:HLposdef}, using Seiberg-Witten theory.
\end{remark}

The proof of Theorem \ref{thm:HLposdef} relies on the Ozsv\'ath-Szab\'o closed 4-manifolds invariant \cite{OSsmooth4}. Their invariant in defined as a composition of certain cobordism maps, which, in certain special cases, may be computed using the mapping cone formula.

%%%%%%%%%%%%%%%%%%%%%%%%%%%%%%%%%%%%%%%%%%%%%%%%%%%%%%%%%%%%%%%%%%%%%
%    
% To add references to your document, replace the two \bib commands below. 
%
%         1. You can use a list of \bib commands for the items you reference as is
%         done in our toy example here.
%
%         2. A second option is to use the command 
%             \bibselect{yourltbfile}
%         to point to a file of \bib commands that should be named 
%         yourltbfile.ltb and be placed in the same folder as your LaTeX
%         source files. 
%
%         3. A third option is to use the command 
%             \bibliography{yourbibfile}
%         to point to a file of BibTeX \bib commands that should be named 
%         yourltbfile.bbl and be placed in the same folder as your LaTeX
%         source files. 
%   
% If you use option 3. above, you should comment out or delete the lines
%            \begin{bibdiv}
%                \begin{biblist}
%        before the \bib command below as well as the line
%                  \end{biblist}
%              \end{bibdiv}
%        after it. 
%
% If you use options 2. or 3. and wish to make your source file self-contained you may
%         for final submission, simply copy the \bib entries to your \LaTeX\ file and
%         wrap them, if necessary, as indicated above.
%  
%%%%%%%%%%%%%%%%%%%%%%%%%%%%%%%%%%%%%%%%%%%%%%%%%%%%%%%%%%%%%%%%%%%%%

\bibspread

\bibliography{bib}

\end{document}